\documentclass[11pt]{amsart}
\usepackage[margin=1.19in]{geometry}

\usepackage[font=small,labelfont=bf,
   justification=justified,
   format=plain]{caption}
\usepackage{verbatim}

\usepackage[utf8]{inputenc}
\usepackage{amsmath, amssymb, bbold, graphicx, enumitem, subcaption}
\usepackage{wrapfig}

\usepackage[pdftex,dvipsnames]{xcolor}
\definecolor{darkgreen}{rgb}{0,0.50,0} 
\definecolor{darkred}{rgb}{0.55,0,0}
\definecolor{darkblue}{rgb}{0,0,0.6}
\definecolor{darkteal}{rgb}{0.0, 0.25, 0.5}
\usepackage[pdfborder=0,pagebackref,colorlinks,citecolor=darkgreen,linkcolor=darkteal,urlcolor=darkblue]{hyperref}
\usepackage{adjustbox}

\usepackage{tikz, tikz-cd}%
\usetikzlibrary{matrix,arrows,decorations.pathmorphing,cd,patterns,calc,backgrounds,patterns}
\tikzset{dummy/.style= {circle,fill,draw,inner sep=0pt,minimum size=1.2mm}}
\tikzset{vertex/.style={fill, circle, minimum size=.1cm, inner sep=0pt}}
\usetikzlibrary{shapes.geometric}
\usepackage[capitalise]{cleveref}

\newcommand{\newrefformat}[2]{}

\crefname{lemma}{Lemma}{Lemmas}
\crefname{theorem}{Theorem}{Theorems}
\crefname{definition}{Definition}{Definitions}
\crefname{proposition}{Proposition}{Propositions}
\crefname{remark}{Remark}{Remarks}
\crefname{corollary}{Corollary}{Corollaries}
\crefname{equation}{Equation}{Equations}
\crefname{ex}{Example}{Examples}
\crefname{appsec}{Appendix}{Appendices}
\theoremstyle{plain}
\newtheorem*{theorem*}{Theorem}
\newtheorem{theorem}{Theorem}[section]
\newtheorem{lemma}[theorem]{Lemma}
\newtheorem{corollary}[theorem]{Corollary}
\newtheorem{proposition}[theorem]{Proposition}

\theoremstyle{definition}
\newtheorem{definition}[theorem]{Definition}
\newtheorem{example}[theorem]{Example}

\theoremstyle{remark}
\newtheorem{remark}[theorem]{Remark}

\numberwithin{equation}{section}

\DeclareMathOperator{\Vect}{Vect}

\DeclareMathOperator{\Cob}{Cob}

\newcommand{\bndry}{\partial}
\newcommand{\Hom}{\mathrm{Hom}}
\newcommand{\crit}{\mathrm{Crit}}

\newcommand{\cyl}{{\rm {Cyl}}}

\newcommand{\bind}{{\rm {ind_b}}}
\newcommand{\dind}{{\rm {ind_d}}}
\newcommand{\id}{{\mathrm {id}}}
\newcommand{\tw}{{\mathrm {tw}}}

\newcommand{\bi}{\mathbf{b}}
\newcommand{\de}{\mathbf{d}}
\newcommand{\TT}{{\rm {\tau}}}
\newcommand{\br}{{\rm {\beta}}}
\renewcommand{\emptyset}{\varnothing}

\renewcommand{\Vect}{\mathrm{Vect}}

\newcommand{\RR}{\mathbb{R}}

\newcommand{\cat}[1]{\mathcal{#1}}

\newcommand{\bbr}{\mathbb{R}}

\newcommand{\calc}{\mathcal{C}}

\newcommand{\cals}{\mathcal{S}}

\usepackage[normalem]{ulem}
\newcommand{\out}{\bgroup\markoverwith
	{\textcolor{red}{\rule[.6ex]{3pt}{0.6pt}}}\ULon}

\usepackage[colorinlistoftodos,prependcaption,textwidth=25mm,textsize=tiny]{todonotes}
\usepackage{xargs}                      

\newcommandx{\note}[2][1=]{\todo[color=orange!50!white,linecolor=orange!40!black,size=\tiny]{#2}}
\newcommandx{\noteil}[2][1=]{\todo[inline,color=orange!50!white,linecolor=orange!40!black,size=\normalsize]{#2}}

\title[Nested cobordisms, Cyl-objects and Temperley-Lieb algebras]{Nested cobordisms, Cyl-objects and\\Temperley-Lieb algebras}

\author[Calle]{Maxine E.\ Calle}
\address{Department of Mathematics, University of Pennsylvania,
209 South 33rd Street,
Philadelphia, PA, 19104, USA}
\email{callem@sas.upenn.edu}

\author[Hoekzema]{Renee S.\ Hoekzema}
\address{
Department of Mathematics,
Vrije Universiteit Amsterdam,
De Boelelaan 1111,
1081 HV Amsterdam,
The Netherlands
}
\email{r.s.hoekzema@vu.nl}

\author[Murray]{Laura Murray}
\address{Department of Mathematics \& Computer Science, Providence College, 1 Cunningham Square, Providence, RI, 02918, USA}
\email{lmurray7@providence.edu}

\author[Pacheco-Tallaj]{Natalia Pacheco-Tallaj}
\address{Department of Mathematics, Massachusetts Institute of Technology, 77 Massachusetts Avenue, Cambridge, MA 02139-4307, USA}
\email{nataliap@mit.edu}

\author[Rovi]{Carmen Rovi}
\address{Department of Mathematics and Statistics, Loyola University Chicago, 1032 W. Sheridan Road, IL 60660, USA}
\email{crovi@luc.edu}

\author[Sridhar-Shapiro]{Shruthi Sridhar-Shapiro}
\address{Department of Mathematics and Computer Science, Santa Clara University, 500 El Camino Real, Santa Clara, CA  95053, USA }
\email{ssridhar@scu.edu}

\keywords{Nested manifolds, cobordism categories, cyclic objects, Temperley-Lieb algebras, stratified Morse theory}

\subjclass[2020]{
57R90, 
55N22, 
18F99, 
57K16} 

\begin{document}

\maketitle
\begin{abstract}
We introduce a discrete cobordism category for nested manifolds and nested cobordisms between them.
A variation of stratified Morse theory applies in this case, and yields generators for a general nested cobordism category.
Restricting to a low-dimensional example of the ``striped cylinder''
cobordism category $\cyl$, we give a complete set of relations for the generators.
With an eye towards the study of TQFTs defined on a nested cobordism category, we describe functors $\cyl\to\calc$, which we call $\cyl$-objects in $\cat C$, and show that they are related to known algebraic structures such as Temperley-Lieb algebras and cyclic objects.
We moreover define novel algebraic constructions inspired by the structure of $\cyl$-objects, namely a doubling construction on cyclic objects analogous to edgewise subdivision, and a cylindrical bar construction on self-dual objects in a monoidal category.
\end{abstract}

\tableofcontents
\section{Introduction}
The central objects of study in this paper are \textit{nested manifolds} and cobordisms between them. A nested manifold is a manifold together with a subset that is diffeomorphic to an embedded submanifold (which itself possibly comes with an embedded subsubmanifold, and so on), such that every subsequent embedding has codimension at least 1. Examples of nested manifolds appear throughout geometry and topology, such as knots, configuration spaces, and tangles. 

A \textit{cobordism} between two nested manifolds $M_I$ and $N_I$, for $I$ a sequence of dimensions, is witnessed by a nested manifold $W_{I+1}$ of one dimension higher so that $\bndry W_{I+1} \cong M_I \amalg N_I$. That is, $W_{I+1}$ is a cobordism of ``all the levels at once,'' meaning that it provides a cobordism of each $d$-dimensional embedded submanifold and these cobordisms form a nested manifold themselves.
Cobordism groups of nested manifolds were studied in
\cite{wall1961cobordism, Stong-cobordism-of-pairs}
and a topological cobordism category for nested manifolds was introduced in \cite{ayala2008geometric,Hoekzema}. The aim of this paper and subsequent work is to construct and study a discrete cobordism category of nested manifolds, $\Cob_I$, and functors out of it.

\begin{theorem}
Every morphism in $\Cob_I$ has a Cerf decomposition into elementary nested cobordisms. These elementary nested cobordisms are determined by a nested Morse function, and either have:
\begin{itemize}
    \item no critical points, in which case the elementary cobordism is a mapping cylinder of a nested pseudo-isotopy class of self-diffeomorphisms of the boundary manifold;
    \item one critical point, in which case the critical point $p_j^i$ is an index $j$ critical point on the $d_i$-dimensional submanifold of the nested cobordism.
\end{itemize}
\end{theorem}

The connected components of $\Cob_I$ are equivalence classes of nested manifolds up to nested cobordism and this collection forms a group under disjoint union. Wall \cite{wall1961cobordism} showed that nested cobordism groups split as a direct sum of the cobordism groups in the dimensions involved.

A more classical notion of cobordism of nested manifolds leaves the background manifold invariant. 
In this setting, nested manifolds $(M,N)$ and $(M',N')$, where $N, N'$ are the respective submanifolds, are considered cobordant if $M'$ is diffeomorphic to $M$ and there is a cobordism $V$ between $N$ and $N'$ that embeds into $M\times I$, respecting the embeddings on either boundary. 
A celebrated theorem by Pontryagin, later extended by Thom, shows that cobordism classes of framed $k$-manifolds inside a background manifold $M^m$, where the background cobordism is cylindrical, is given by homotopy classes of maps from $M$ to $S^{m-k}$  \cite{pontryagin55}.
In particular this implies that cobordism groups of all framed manifolds (which can be thought of as sitting inside a large sphere) are isomorphic to the stable homotopy groups of spheres.
A topological cobordism category of nested manifolds inside a fixed background manifold was studied in \cite{randal2011embedded}.

In this paper, we will study these cylindrical background cobordisms within $\Cob_{1<2}$, the nested cobordism category in which objects are circles with marked points and morphisms are surfaces decorated with lines connecting the points. In particular, the objects of this subcategory $\cyl$ are circles with marked points and morphisms are $1$-dimensional cobordisms on a cylinder, which we call ``striped cylinders.'' 
We further simplify this category by quotienting out contractible circles.
Using methods similar to those of \cite{kock_2003, penneys}, we give a generators and relations presentation of $\cyl$.

\begin{theorem}[\cref{thm: gens of cyl}, \cref{thm: minimal list of generators}, \cref{cor: relns in cyl}]
The objects of $\cyl$ are generated by circles with marked points, $S^1_k$, with points labeled $0,1\dots, k-1$. The morphisms are generated by striped cylinders that are the identity, have twisted stripes ($\tw_k$), have a birth at marked point $i$ ($\bi^i_k$) or have a death at marked point $i$ ($\de^i_k$); see~\cref{intro: gens}.

\begin{figure}[h!]
     \centering
     \begin{subfigure}[b]{0.2\textwidth}
         \centering
         \includegraphics[width=\textwidth, trim={3cm 21.5cm 12cm 3.5cm},clip]{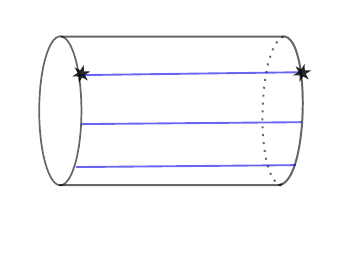}
         \caption{$\id_k$}
         \label{fig: identity}
     \end{subfigure}
     \begin{subfigure}[b]{0.2\textwidth}
         \centering
         \includegraphics[width=\textwidth, trim={11cm 21.5cm 4cm 3.5cm},clip]{Generators.pdf}
         \caption{$\tw_k$}
         \label{fig: twist}
     \end{subfigure}
     \begin{subfigure}[b]{0.2\textwidth}
         \centering
         \includegraphics[width=\textwidth, trim={3cm 14.2cm 12cm 10cm},clip]{Generators.pdf}
         \caption{$\bi_k^i$}
         \label{fig: birth}
     \end{subfigure}
     \begin{subfigure}[b]{0.2\textwidth}
         \centering
         \includegraphics[width=\textwidth, trim={11cm 14.2cm 4cm 10cm},clip]{Generators.pdf}
         \caption{$\de_k^i$}
         \label{fig: death}
     \end{subfigure}
        \caption{Generating cobordisms}
	\label{intro: gens}
\end{figure}
A complete description of the relations is given in~\cref{thm: minimal list of generators}, which includes the usual relations on 1-dimensional cobordisms (the snake relation, etc.) as well as relations involving how the twist interacts with the birth and death cobordisms.
\end{theorem}

The motivation for providing a generators and relations description of $\cyl$ (and $\Cob_{1<2}$ more generally) is to understand the explicit data needed to construct functors out of these cobordism categories. As a consequence of the previous theorem, we obtain such a description for a functor $\cyl\to \cat C$, where $\cat C$ is any category. We call these functors \emph{$\cyl$-objects} in $\cat C$.

\begin{corollary}[{\cref{cor:cyl rep data}}]
A $\cyl$-object in $\cat C$ is specified by the following data:\begin{itemize}
        \item for each $n\geq 0$, an object $c_n\in \cat C$,
        \item for each $n\geq 0$, an isomorphism $t_n\colon c_n\to c_n$,
        \item for each $n\geq 2$, maps $d_n^i\colon c_n\to c_{n-2}$ for $0\leq i\leq n-1$,
        \item for each $n\geq 0$, maps $s_n^j\colon c_n \to c_{n+2}$ for $0\leq j \leq n+1$,
    \end{itemize}
    subject to the relations
    \begin{enumerate}
        \item[(i)] $d_{k-2}^i\circ d_{k}^j = d_{k-2}^{j-2}\circ d_{k}^{i} $ for $ i < j - 1 $,
        \item[(ii)] $s_{k+2}^i\circ s_{k}^j=s_{k+2}^{j+2}\circ 
        s_k^{i}$ if $i \leq j$,
        \item[(iii)] $d_{n+2}^j\circ s_n^i = \left\{\begin{array}{cc}
             \id & i=j-1,j,j+1; \\
             s_{n-2}^{j-2} \circ d^i_n & i<j-1;\\
             s_{n-2}^j \circ d_n^{i-2} & i> j+1,
        \end{array}\right.$
        \item[(iv)] $t_n^n = \id$,
        \item[(v)] $t_{n+2}\circ s_n^j = s_n^{j+1}\circ t_n$,
        \item[(vi)] $t_n\circ d_{n+2}^i = d_{n+2}^{i+1}\circ t_{n+2}$.
    \end{enumerate}
\end{corollary}

This definition is reminiscent of that of a cyclic object \cite{connes, loday:92}. In \cref{sec:cyclic objects}, we detail this connection and show the following result.

\begin{theorem}[{\cref{thm:cyl and cyclic cat}, \cref{thm:inclusion of c2 cyc into cyl}, \cref{cor:cyl0 obj gives c2 cyclic}}]
    There is an inclusion of the cyclic category into $\cyl$, and consequently every functor $\cyl\to \cat C$ determines a cyclic object.
\end{theorem}

Inspired by the cyclic bar construction, we define a ``cylinder bar construction'' for finite-dimensional vector spaces (see \cref{defn:bar construction}), which is a $\cyl$-object. We expect this construction to give rise to interesting algebraic structures and
plan to study it further in future work. 
We also show that $\cyl$-objects are closely related to representations of affine Temperley-Lieb algebras \cite{Graham-Lehrer:98, fan/green:97, green:98, erdmann/green:98} and annular Temperley-Lieb algebras \cite{jones:01, Jones, penneys}.

\begin{theorem}[{\cref{cor:cyl object is affine TLA}, \cref{cor:cyl0 obj is annular TLA}}]
    Every $\cyl$-object determines an affine Temperley-Lieb algebra and an annular Temperley-Lieb algebra. 
\end{theorem}

\textit{Topological quantum field theories} (TQFTs) are defined as symmetric monoidal functors out of a cobordism category into a linear category such as ${\rm Vect}_k$.
In this light we can think of $\cyl$-objects in $\Vect_k$ as TQFTs on the striped cylinder cobordism category, although we note that $\cyl$
does not have an interesting symmetric monoidal structure.
Within mathematical physics, TQFTs 
(and their generalizations) can be viewed as mathematical models for quantum field theories in which transition amplitudes depend only on topological properties of the system. This occurs, for example, in the case of Chern-Simons theory, 
a central object of study across topology, gauge theory, and representation theory.
TQFTs are increasingly ubiquitous in the theoretical physics literature as well, where they have a wide range of applications including modeling anomalies \cite{Dai-Freed, Witten-anomaly, freed2019lectures} and the low energy behavior of lattice models in condenced matter physics \cite{Freed-Hopkins, walker201131tqfts}.

Within algebraic topology, TQFTs represent information about the geometric gluing structure of manifolds. 
An example of this interpretation is the ``folklore theorem'' giving an equivalence of categories between 2-dimensional TQFTs and commutative Frobenius algebras
over $k$ \cite{Dijkgraaf89, kock_2003, Abrams96}. There are many variants of this theorem in the literature that consider different cobordism categories \cite{schommerpries2014classification, hanbury-open-closed, baas-open-closed, bartlett2015modular} and which play an important role in the physics literature 
\cite{Moorebranelectures, Johnson_Freyd_2022}.
Defining TQFTs on nested cobordism categories enlarges
the connection between algebraic structures and gluing of geometric objects, and could potentially lead to new connections with physical systems.
In upcoming work we will extend our scope to consider the full category $\Cob_{1<2}$ of striped surface cobordisms, aiming to provide a classification of $2$-dimensional nested TQFTs. We expect that this work will be highly related to the study of $2$-dimensional defect TQFTs \cite{Carqueville}. 

\subsection{Outline} 
In \cref{section2} we define the nested cobordism category $\Cob_I$ in full generality and develop a version of Morse theory for nested manifolds by application of stratified Morse theory, which is summarized in \cref{sec: SMT}. We use this nested Morse theory to give a list of generators with zero or one critical point(s) for $\Cob_I$.
In \cref{section3:cyl} we restrict our scope to the category $\cyl$ of striped cylinders. We establish a generators-relations presentation of this category, with the use of topological invariants and the factorization of a morphism into a unique normal form.
In \cref{section6:TQFTs}, we use the generators and relations of $\cyl$ to give a full description of the data needed to build a $\cyl$-object in a general category $\cat C$, and we discuss the connection to affine and annular Temperley-Lieb algebras (\cref{section5:TLalgebras} and \cref{sec:annular TLA}) and cyclic objects (\cref{sec:cyclic objects}) as well as defining the doubling and cylindrical bar constructions (\cref{sec:cyclic objects,sec:bar construction}).

\subsection{Acknowledgements}
This paper began in conversations at the 2023 Women in Topology IV Workshop and we thank the organizers of this program and the Hausdorff Research Institute for Mathematics for their hospitality during the workshop. 
We would also like to thank the Foundation Compositio Mathematica, the Foundation Nagoya Mathematical Journal and the $K$-theory Foundation for financial support for this event. We are also grateful to the American Institute of Mathematics' SQuaREs program, whose support was crucial for the completion of this project.
The first-named author was partially supported by the National Science Foundation (NSF) grant DGE-1845298.
The second-named author was supported by the Dutch Research Council (NWO) through the grant VI.Veni.212.170.
The third-named author was partially supported by NSF grant DMS-2316646.
The fourth-named author was partially supported by NSF grant DGE-2141064.
The authors also benefited from conversations with   
\'Alvaro del Pino G\'omez,
Tony Giaquinto,
Inbar Klang,
Peter Kristel,
Alexis Langlois-R\'emillard,
Cary Malkiewich,
Kyle Miller,
Emily Peters,
Thomas Rot,
Mateusz Stroi\'nski,
and
Lauran Toussaint.
We also thank the anonymous referees for useful input.

\section{Nested manifolds and cobordism}
\label{section2}

\subsection{The nested cobordism category $\Cob_I$}
In this section we set up the language and theory of nested manifolds and cobordisms, following \cite{Hoekzema}. A nested manifold can be thought of as a manifold with a collection of lower-dimensional manifolds nicely embedded within it. 

\begin{definition}
    Given a sequence $I=(d_1<\dots<d_n)$ of non-negative integers, a \textit{nested (compact) $I$-manifold} is an ordered tuple\[
    M = (M_{d_n},\dots,M_{d_1})
    \] where $M_{d_n}$ is a smooth, (compact) $d_n$-manifold and for each $1\leq i\leq n-1$, $M_{d_{i}}$ is a closed subset of $M_{d_{i+1}}$ which is diffeomorphic to a smooth $d_{i}$-dimensional manifold.
\end{definition}

\begin{figure}[h!]
	\centering
	\includegraphics[width=0.5\textwidth]{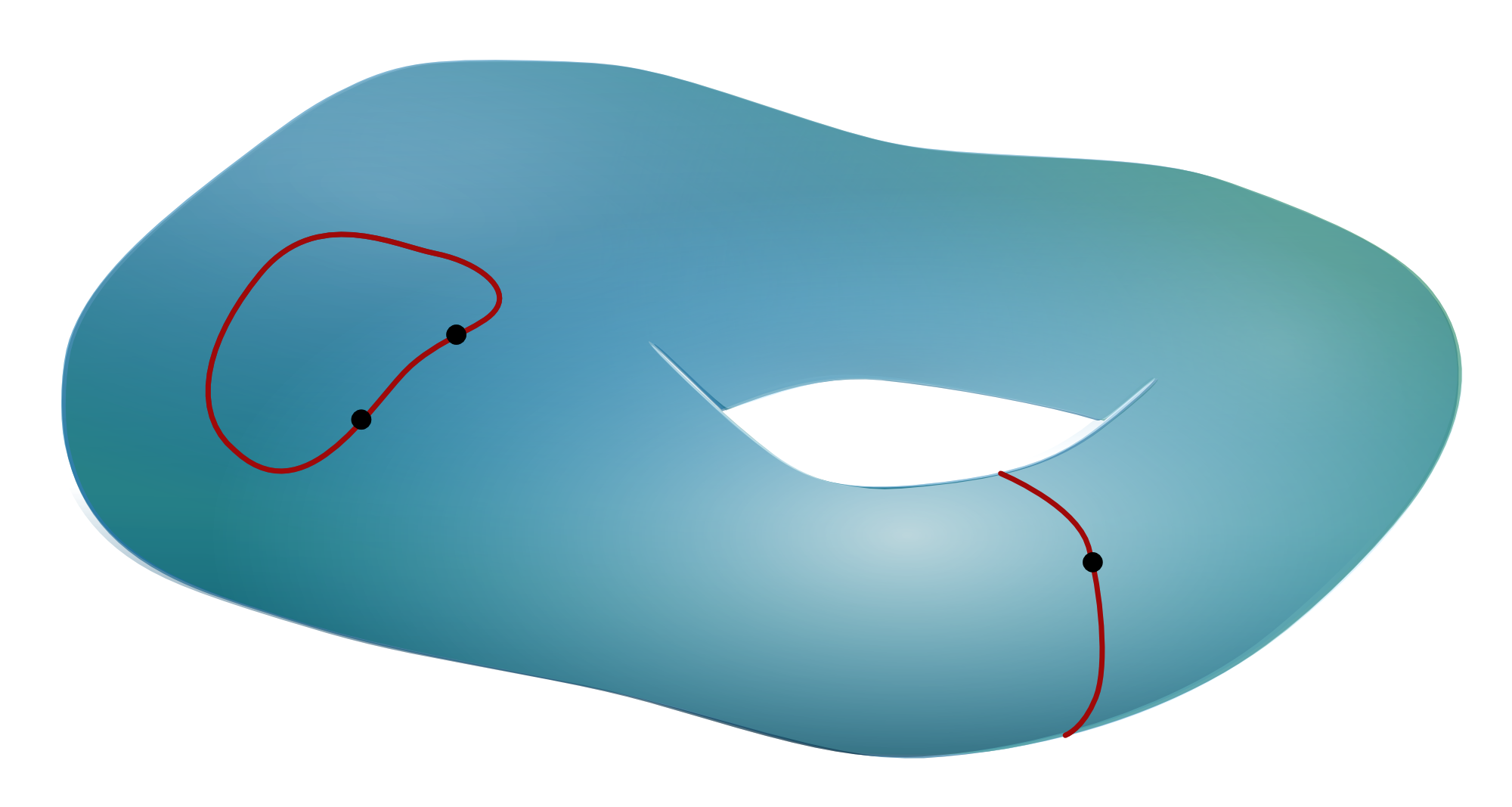}
	\caption{An example of a nested manifold $M_{0<1<2}$: a surface endowed with a 1-dimensional submanifold, which is itself endowed with a 0-dimensional submanifold.}
	\label{manifold}
\end{figure}

Nested manifolds are a special case of stratified manifolds, see~\cref{lem: nested surface is strat}. In particular, stratified manifolds allow for singularities at substrata that are not allowed for nested manifolds; see~\cite{SMT88} for more background on stratified manifolds.

\begin{remark}
    Note that we think of a nested manifold as a manifold together with a sequence of subsets. This means we do not include additional data of the embeddings of the manifolds into each other.
\end{remark}

For the purpose of this paper we will define an orientation on a nested manifold $M=(M_{d_n}, \dots, M_{d_1})$ as an orientation on the top dimensional-manifold $M_{d_n}$. Note that this does not necessarily induce an orientation on the submanifolds $M_{d_i}$.
Other definitions are possible and may be of interest in future work.

\begin{definition}
    A \textit{smooth map} $f\colon M\to M'$ of nested $I$-manifolds is an ordered $I$-tuple of smooth maps $f=(f_n,\dots, f_1)$ so that the following diagram commutes for each $1\leq i\leq n$:
    \[\begin{tikzcd}
        M_{d_{i-1}} \ar[r, "f_{i-1}"] \ar[d] & M'_{d_{i-1}}\ar[d] \\
        M_{d_{i}} \ar[r, swap, "f_i"] & M'_{d_{i}}
    \end{tikzcd}
    \] where the vertical arrows are the inclusion of the submanifolds. In particular, a map $f\colon M_{d_{n}} \rightarrow M'_{d_{n}}$ is a map of nested manifolds if for each $i$ the image of the restriction $f_i = f|_{M_{d_{i}}}$ is contained in $M'_{d_{i}}$.

    We call $f$ a nested diffeomorphism if each $f_i$ is a diffeomorphism. We call $f$ orientation-preserving if each $f_i$ is orientation-preserving.
\end{definition}

We say an $I$-manifold $M$ is \textit{closed} if $M_{d_{i}}$ is closed for all $i$. We will also need $I$-manifolds \textit{with boundary}, which intuitively means that all of the $\bndry M_i$ sit inside $\bndry M_{d_{n}}$ as a nested $I-1$ manifold, where $I-1 :=(d_1-1<d_2-1<\dots<d_n-1)$. 

\begin{definition}
A nested $I$-manifold with boundary is an $I$-manifold $M=(M_{d_{n}},\dots, M_{d_{1}})$ where $M_{d_{n}}$ is a manifold with boundary, such that $\bndry M_{d_i}$ is a topologically closed subset of $\bndry M_{d_{i+1}}$.
We also require that any closed (boundaryless and compact) components of $M_{d_{j}}$ lie in the interior of all higher-dimensional manifolds (i.e. the closed components of $M_{d_{j}}$ do not intersect $\bndry M_{d_{j+k}}$ for any $k>0$).
Moreover, we require the normal bundle of $\bndry M_{d_{i}}\subset M_{d_{i}}$ to be a subbundle of the normal bundle of $\bndry M_{d_{i+1}}\subset M_{d_{i+1}}$ restricted to $\bndry M_{d_{i}}$. 

\end{definition}

\begin{definition}
    We call two closed $(I-1)$-manifolds $M_{I-1}$ and $M'_{I-1}$ \emph{nested cobordant} if there is a compact $I$-manifold $W_{I}$ with boundary such that $\bndry W \cong M \amalg M'$. The nested manifold $W_I$ is said to be a \textit{nested cobordism} between $M$ and $M'$.
    
    If $M$ and $M'$ are oriented nested manifolds, then we call them oriented cobordant if there is an oriented $I$-manifold $W$ such that $\bndry W \cong M \amalg \overline{M'}$ where the nested diffeomorphism is orientation-preserving. Here $\overline{M'}$ denotes the nested manifold $M'$ with the orientation reversed for every $M_{d_i}$. 

The data of an oriented cobordism from $M$ to $M'$, also written as $W_I\colon M_{I-1}\Rightarrow M'_{I-1}$, is the oriented $I$-manifold $W$ along with orientation-preserving nested diffeomorphisms 
    \[\begin{tikzcd}
            M \ar[r, hook] & W & \overline{M'}\ar[l, hook']
    \end{tikzcd}
    \] which map $M$ and $M'$ diffeomorphically (as nested manifolds) onto the in- and out-boundary of $W$, respectively.
\end{definition}

As in the non-nested setting, there are many different cobordisms between two $I$-manifolds  that are diffeomorphic. 

\begin{definition}\label{defn:equiv reln on nested cobs}
    Two nested cobordisms from $M$ to $M'$ are \textit{diffeomorphism equivalent} if there is a diagram\[
    \begin{tikzcd}
        & W \ar[dd, "f"] &\\
        M\ar[ru, hook]\ar[rd, hook] && \overline{M'}\ar[ul, hook'] \ar[ld, hook']\\
        & W' &
    \end{tikzcd}
    \] 
    so that $f$ is a nested diffeomorphism preserving the boundaries pointwise. \end{definition}

\begin{definition}
   We define $\Cob_I$ to be the category with objects closed $(I-1)$-manifolds and morphisms diffeomorphism equivalence classes of nested $I$-cobordisms between the objects.
\end{definition}

In order for this category to be well-defined, we need to show that working with equivalence classes of nested cobordisms also allows us to model composition using pushouts; the following theorem is the analogue of \cite[Theorem 1.4]{milnorhcobordism}, \cite[Theorem 1.3.12]{kock_2003} for composition in the non-nested cobordism category.

\begin{theorem}\label{thm: pushouts give composition}
    Let $W_I\colon M_{I-1}\Rightarrow M'_{I-1}$ and $W'_I\colon M'_{I-1}\Rightarrow N_{I-1}$ be two nested cobordisms. Then up to diffeomorphism we can give the pushout $W\cup_{M'} W'$ the structure of a nested $I$-manifold such that the embeddings \[
    M \hookrightarrow W\cup_{M'} W' \hookleftarrow \overline{N}
    \] are orientation-preserving nested diffeomorphisms onto their images.
\end{theorem}
\begin{proof}
Up to diffeomorphism we can assume the top dimensional cobordisms $W_{d_n}$ and 
$W'_{d_n}$ have a collar at their boundaries, whose restrictions forms collar neighbourhoods for every nested submanifold. That is, there exists an $\varepsilon>0$ such that $(1-\varepsilon, 1]\times M'_{I-1} \subset W_{I}$ and $[0, \varepsilon) \times M'_{I-1} \subset W'_{I}$. In particular this implies that the submanifolds are cylindrical in the collar as well, so the pushout defined as 
$$W\cup_{M'} W = (W_{d_n}\cup_{M'_{d_n-1}} W'_{d_n}, \dots, W_{d_1}\cup_{M'_{d_1-1}} W'_{d_1})$$
inherits a smooth structure at every level.\end{proof}

This result implies that pushouts yield a well-defined composition in $\Cob_I$.
It follows from the definition that composition is associative and that cobordisms that are nested diffeomorphic to $W_I = M_{I-1}\times [0,1]$, with the boundary inclusions being identities, are identity morphisms in the category. 

\begin{proposition}\label{propn: diffeo gives cobord}
    Any nested diffeomorphism $\phi\colon M'\to M$ determines a nested cobordism from $M$ to $M'$ that is an isomorphism in $\Cob_I$ given by the mapping cylinder $M_\phi$ of $\phi$.
\end{proposition}
\begin{proof}
Consider the cobordism $W = M_{I-1}\times [0,1]$ with inclusion maps
    \[\begin{tikzcd}
            M \ar[r, "id", hook] & W & \overline{M'}\ar[l,"\overline{\phi}"' , hook'].
    \end{tikzcd}\]
        This has an inverse given by 
        \[\begin{tikzcd}
            M' \ar[r, "\phi", hook] & W & \overline{M}\ar[l,"\overline{id}"' , hook'].
        \end{tikzcd}\]  
        since their composition is $W\cup_{\phi^{-1}\circ \phi}W \cong W\cup_{\text{id}}W \cong M\times [0,2]$ which is diffeomorphic to the trivial product cobordism 
        \[\begin{tikzcd}
            M \ar[r, "id", hook] & M\times I & \overline{M}\ar[l,"\overline{id}"' , hook'].
        \end{tikzcd}\]  
        by a diffeomorphism that shrinks the interval.
\end{proof}

\begin{remark}\label{rmk: diffeo up to isotopy}
    By \cref{propn: diffeo gives cobord}, nested diffeomorphic manifolds are isomorphic as objects in the category. Since a category is equivalent to its skeleton, we can think of $\Cob_I$ as having objects given by diffeomorphism classes of $I-1$ manifolds. 
           \end{remark}
\begin{definition}
We call nested diffeomorphisms $\phi, \psi\colon M \rightarrow M$ \emph{nested pseudo-isotopic} if there is a nested diffeomorphism $F\colon M\times I \to M\times I$ such that $F|_{M\times \{0\}} = \phi$ and $F|_{M\times \{1\}} = \psi$. 
\end{definition}

The lemma below is the nested analogue of \cite[Theorem 1.9]{milnorhcobordism}.
\begin{lemma}\label{cor: automorphisms are diffeos up to pseudo-isotopy}
Two mapping cylinders of nested self-diffeomorphisms $\phi, \psi\colon M \rightarrow M$ are equivalent as morphisms in $Cob_I$ if and only if $\phi$ is nested pseudo-isotopic to $\psi$. 
\end{lemma}

\begin{proof}
The maps $\phi$ and $\psi$ are pseudo-isotopic if and only if $\phi^{-1} \circ \psi$ is pseudo-isotopic to the identity. 
Composing $M_\phi$ with $M^{-1}_\psi$, where the latter is given by
     \[\begin{tikzcd}
            M \ar[r, "\psi", hook] & M\times I & \overline{M}\ar[l,"\overline{id}"' , hook'],           \end{tikzcd}\]
gives $W_1\cup_{\phi^{-1}\circ \psi}W_2$ with the inclusions on either end the identity and $W_{1,2}=M\times I$. 
Let $F\colon M\times I\to M \times I$ be a nested pseudo-isotopy between $\phi^{-1}\circ \psi$ and $id$.
Then $\widetilde{F}\colon W_1\cup_{\phi^{-1}\circ \psi}W_2 \to M\times [0,2]$ defined as $F$ on $W_1$ and $id$ on $W_2$ is a nested diffeomorphism relative boundary between $M_\phi\circ M^{-1}_\psi$ and the identity.

Conversely, if $G\colon M\times I \to M\times I$ is a nested pseudo-isotopy between $\phi$ and $\psi$, i.e. $G|_{M\times \{0\}} = \phi$ and $G|_{M\times \{0\}} = \psi$.
Consider the morphism $W_{\phi, \psi}$ defined as
     \[\begin{tikzcd}
            M \ar[r, "\phi", hook] & M\times I & \overline{M}\ar[l,"\overline{\psi}"' , hook'],           \end{tikzcd}\]
Note that by definition the composite $id \circ W_{\phi, \psi} \circ id = M_{\phi^{-1}}\circ M_{\psi} \circ id \cong M^{-1}_{\phi}\circ M_{\psi}$.
Consider $\widetilde{G}\colon M\times [0,3] \to id \circ W_{\phi, \psi} \circ id $ defined by $G$ on the middle cylinder and the identity elsewhere. This is a nested diffeomorphism relative boundary that witnesses that $M_\phi$ is inverse to $M_\psi$.
\end{proof}           

\subsection{Nested Morse Theory}
\label{section3:nestedMorse}
Following methods of \cite{kock_2003}, we will use Morse theoretic arguments to find generators of $\Cob_I$. In this subsection, we develop helpful tools for nested cobordisms, using arguments analogous to those from non-nested Morse theory \cite{audin2014morse} as well as some results from stratified Morse theory \cite{SMT88}. See Appendix~\ref{sec: SMT} for a review of stratified Morse theory in the more general setting.  

Let $M_I = (M_{d_n}, \dots, M_{d_1})$ be a nested $I$-manifold, possibly with boundary.
Let $f_n\colon M_{d_n}\to \RR$ be a smooth function and $f_i=(f_n)|_{M_{d_i}}$.
We denote the set $f=(f_n, \dots, f_1)$ and call $f$ a \textit{nested function}.

\begin{definition}\label{def: indiv/nested Morse}
    A nested function $f\colon M_I\to \RR$ is \textit{individually Morse} if all the $f_i$ are Morse, i.e. $f_n$ is proper and each $f_i$ has non-degenerate critical points and distinct critical values. We call $f$ \textit{nested Morse} if moreover the critical points $\crit(f_i)$ of the $f_i$ are distinct, i.e. $\crit(f_i) \cap \crit(f_j) = \varnothing$ if $i\neq j$.  
    Denote $\crit(f) = \bigcup_i \crit(f_i)$.
\end{definition}

\begin{figure}[h!]
	\centering	\includegraphics[width=0.5\textwidth]{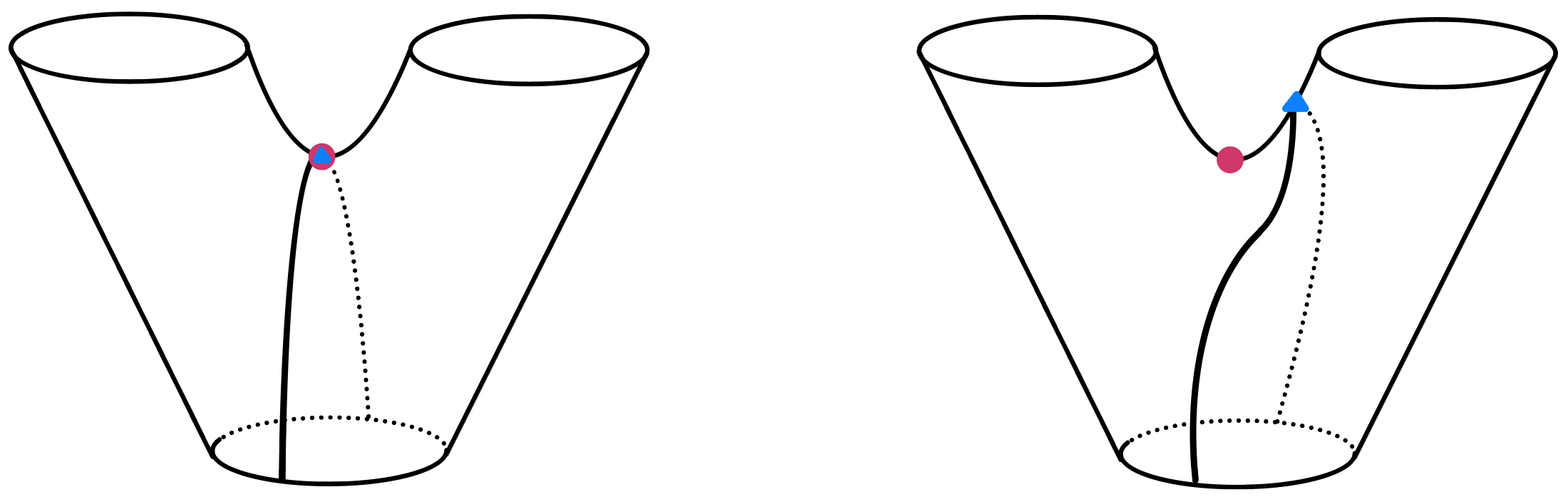}
	\caption{The figure on the left is an example of an individually Morse function which is not nested Morse. The figure on the right is nested Morse.}
	\label{morse nested morse}
\end{figure}

We claim that every nested function can be approximated by a nested Morse function (Theorem~\ref{thm: nested Morse dense in smooth}).

\begin{lemma}\label{lem:distance to p is nested morse}[See \cite{audin2014morse}, Proposition 1.2.1]
Given an embedding of $M_I$ into $\bbr^N$, we have that for almost all $p\in \bbr^N$, the function 
\[
f_p: M_I \to \bbr, \hspace{1cm} x \mapsto ||x-p||^2
\]
is a nested Morse function.  
\end{lemma}

\begin{proof}
For each $i\in I$, we write $f_p^i:= f_p|_{M_i}$. First observe that the collection of $p\in \RR^N$ for which $f_p$ fails to be nested Morse can be written as a union\[
\{p\in \RR^N\mid f_p \text{ not individually Morse}\} \cup \{p\in \RR^N\mid \crit(f^i_p) \cap \crit(f_p^j)\neq \varnothing \text{ for some } i,j\in I\}.
\] We will show that both sets in the union above have measure zero in $\RR^N$, thereby proving the claim that $f_p$ is nested Morse for almost all choices of $p$.

For the first set, by \cite[Proposition 1.2.1]{audin2014morse}, Sard's theorem ensures that for any $i$, the function $f_p^i$ is Morse except for a measure zero set of values $p$. The union of these sets for all $i=1, \dots, n$ is still a measure zero set, hence the same proof shows that almost all $f_p$ maps are individually Morse on $M_I$.

For the second set, we will now show that $A_{ij} = \{p\in \RR^N\mid \crit(f^i_p)\cap \crit(f^j_p)\neq \varnothing\}$ has measure zero for any choice of $i\neq j\in I$. Consequently, the union $\bigcup_{i,j\in I} A_{ij}$ also has measure zero (since $I$ is finite), which proves the desired result. Without loss of generality, we may assume that $i<j$. Recall (c.f. \cite[\S 1.2 and \S 1.4]{audin2014morse}) that $x\in \crit(f_p^i)$ if and only if $(x-p)\perp T_xM_i$, i.e. $x-p\in N_xM_i$, where the tangent space is taken with respect to $M_i\subseteq M_n\hookrightarrow \RR^N$. Consequently, $x\in \crit(f_p^i)\cap \crit(f_p^j)$ if and only if $x\in M_i\subseteq M_j$ and $v:=(p-x)\in N_x M_j \subseteq N_xM_i$, which is to say $(x,v)\in M_i\times_{M_j} NM_j\subseteq NM_j$. Consider the smooth map $E_j\colon NM_j\to \RR^N$ from \cite[\S 1.2.a]{audin2014morse} which sends $(x,v)\mapsto x+v$. Let $E_{ij} := E_j|_{M_i\times_{M_j} NM_j}$ and observe that $A_{ij} = {\rm im}(E_{ij})$. Now, since $\dim(M_i\times_{M_j} NM_j) = d_i + N-d_j = N - (d_j - d_i)$, the image of $E_{ij}$ has dimension strictly less than $N$ (since $d_i<d_j$) and hence $A_{ij}$ has measure zero.
\end{proof}

\begin{theorem}\label{thm: nested Morse dense in smooth}
    Every smooth function $f\colon M_I\to \bbr$ can be uniformly approximated by a nested Morse function on any compact subset.
\end{theorem}

\begin{proof}
Let $f=(f_1,\dots, f_n)\colon M_I\to \RR$ be a smooth function. Then, as in the proof of \cite[Corollary 6.8]{milnor} (or \cite[Proposition 1.2.4]{audin2014morse}), we may choose an embedding $h$ of $M_{d_n}$ into $\RR^N$ for $N$ sufficiently large so that the first coordinate of $h$ is $f_n$. By \cref{lem:distance to p is nested morse}, for almost any point $p=(-c+\varepsilon_1, \varepsilon_2, \dots, \varepsilon_n)$ near $(-c, 0, \dots, 0)$, the function $f_p$ is not only Morse, but nested Morse. Consequently, the function\[
g(x) = \frac{f_p(x) - c^2}{2c}
\] is also  nested Morse. It follows (as in the non-nested case) that for $c$ sufficiently large and $\varepsilon_i$ sufficiently small, the function $g(x)$ is a uniform approximation of $f$ on any compact subset. 
\end{proof}

\begin{remark}
The result above is similar to \cite[Theorem 2.2.1]{SMT88}, which gives this argument for the case of analytic manifolds; see Theorem~\ref{thm: SMT dense in smooth maps}, which recalls this result for stratified Morse functions on subanalytic manifolds. It is possible to apply the stratified result in the nested setting, citing the fact that a compact manifold admits an analytic structure (c.f. \cite{shiga}) and taking care that this structure is suitably compatible with the stratification coming from the nested structure. We thank the anonymous referee for highlighting the subtlety of this approach and suggesting the proof method detailed above.
\end{remark}

Stratified Morse theory also provides an explicit description of how the nested (or stratified) manifold changes as one moves past critical points.
It will be helpful to use a dictionary between stratified Morse theory and our nested Morse functions in order to describe how our nested manifolds change as one moves past critical points.

More specifically, for $M_I$ a nested manifold, let $F_i = M_{d_i}\setminus M_{d_{i-1}}$ for $1<i\leq n$ and $F_1 = M_{d_1}$. Note that $F_i$ are manifolds of dimension $d_i$; these will be the strata of a Whitney stratified space $F = (F_1, \dots, F_n)$, as shown below. Note that in this situation the stratified space $Z = \bigcup_i F_i = M_{d_n}$ is the entire background manifold.

\begin{lemma}\label{lem: nested surface is strat}
For any $M_{I}$, $F$ is a Whitney stratified space in the sense of \cite{SMT88}.
    \end{lemma}
\begin{proof}
    Note first that $F_i$ is a locally closed smooth submanifold of $M=M_{d_n}$ of dimension $d_i$. We will show that every pair $F_\alpha$ and $F_\beta$ for $\alpha < \beta$ satisfies Whitney conditions A and B. Suppose $x_i \in F_\beta$ converges to $y \in F_\alpha$, and $y_i \in F_\alpha$ also converges to $y$. In a local coordinate system on $M$, the secant lines $l_i = \overline{x_i y_i}$ converge to a line $l\subset T_y M$ and the tangent planes $T_{x_i} F_{\beta}$ converge to a plane $\tau \subset T_y M$. We need to show that 
    \begin{enumerate}
        \item[(a)] $T_y F_\alpha \subset \tau$ 
        \item[(b)] $l\subset \tau$.
    \end{enumerate}
Consider $M_{d_\beta} = \bigcup_{i<\beta} F_i$. Then $\tau = T_y M_{d_\beta} \supset T_y F_\alpha$, and we can form the secant lines $l'_i = \overline{x_i y_i}$ in a local coordinate system of $M_{d_\beta}$ instead, where they will have the same limit $l \in T_y M_{d_\beta} =\tau$ as $M_{d_\beta}$ is a submanifold of $M_{d_n}$.
\end{proof}

\begin{lemma}\label{lem: nested Morse iff strat Morse}
    Let $M_I$ be a nested manifold and let $f_n\colon M_{d_n}\to \bbr$ be a smooth function. The nested function $f=(f_n, \dots, f_1)\colon M_I\to \bbr$ is nested Morse if and only if $f_n$ is a stratified Morse function (see Definition~\ref{def: strat Morse func}) on $F=(F_1, \dots, F_n)$.
\end{lemma}

\begin{proof}
Assume $f$ is a nested Morse function. 
By definition, $f_n$ is proper and has distinct critical values, and all the critical points of $f_i:= f_n|_{M_{d_i}}$ are non-degenerate and distinct, therefore so are the critical points of the further restrictions $f_i|_{M_{d_i}\setminus M_{d_{i-1}}} = f_i|_{F_i}$.
It remains to show that for any critical point, the only (generalized) tangent space that is annihilated at that point is that of the stratum containing the point.
Let $p$ be a critical point on the stratum $F_i = M_{d_i}\setminus M_{d_{i-1}}$, so the (generalized) tangent space at $p$ of $F_i$, which is $T_p M_{d_i}$, is annihilated by $df_i$.
For $j>i$, we have that the generalized tangent space to $F_j$ at $p$ is simply $T_p M_{d_j}$. The nested Morse condition states that if $p$ is a critical point of $M_{d_i}$, it is not also a critical point of $M_{d_j}$ for $j \neq i$. Hence $T_p M_{d_j}$ is not annihilated by $df_n$. So a nested Morse function $f$ gives a stratified Morse function $f_n$ on $F$.

Conversely, assume that $f_n$ is a stratified Morse function. 
By definition, $f_n$ is proper with distinct critical values, and for all strata $F_i$, the critical points are non-degenerate. 
Suppose $p$ is a critical point of $F_i$. As above, the generalized tangent space to $F_j$ at $p$ is $T_p M_{d_j}$ for $j\geq i$. By the generalized tangent space condition, $T_p M_i$ is the only generalized tangent space that is in the kernel of $df_n$. Hence $p$ is not a critical point of $M_{d_j}$ for $j\neq i$. It follows that
    \begin{itemize}
        \item[(A)] The critical points of $f_i\colon M_{d_i} \rightarrow \bbr$ lie in the interior of $M_{d_i}\setminus M_{d_{i-1}}$, i.e. they equal the critical points of $f_n$ restricted to $F_i = M_{d_i}\setminus M_{d_{i-1}}$, which are non-degenerate. Hence $f_i$ is Morse on all of $M_{d_i}$. Therefore $f = (f_n, \dots,  f_1)$ is individually Morse.
        \item[(B)] The critical points of $f_i$ are disjoint, and hence $f$ is nested Morse.\qedhere
    \end{itemize}
\end{proof}

These lemmas allow us to carry over definitions from stratified Morse theory into our nested manifold setting, particularly in regards to the way a nested manifolds changes as one moves past critical points.

\begin{definition}[\cite{SMT88}, Definition I.3.3]\label{defn: Morse data}
    Fix $\epsilon>0$ so that the interval $[v-\epsilon, v+\epsilon]$ contains no critical values of $f\colon M_{I}\to\bbr$ other than $v=f(p)$.
     A pair $(A,B)$ of stratified spaces is \emph{Morse data} for $f$ at $p$ if there is an embedding $h\colon B\to (M_{I})_{\leq v-\epsilon}$ such that $(M_{I})_{\leq v+\epsilon}$ is homeomorphic to $(M_{I})_{\leq v-\epsilon}\cup_B A$ (obtained by attaching $A$ along $B$ using the attaching map $h$). The homeomorphism preserves the stratification.
\end{definition}

\begin{figure}[ht]
\adjustbox{scale=0.65}{

\tikzset{every picture/.style={line width=0.75pt}} 
\begin{tikzpicture}[x=0.75pt,y=0.75pt,yscale=-1,xscale=1]

\draw (345.5,443.74) node  {\includegraphics[width=417.75pt,height=190.1pt]{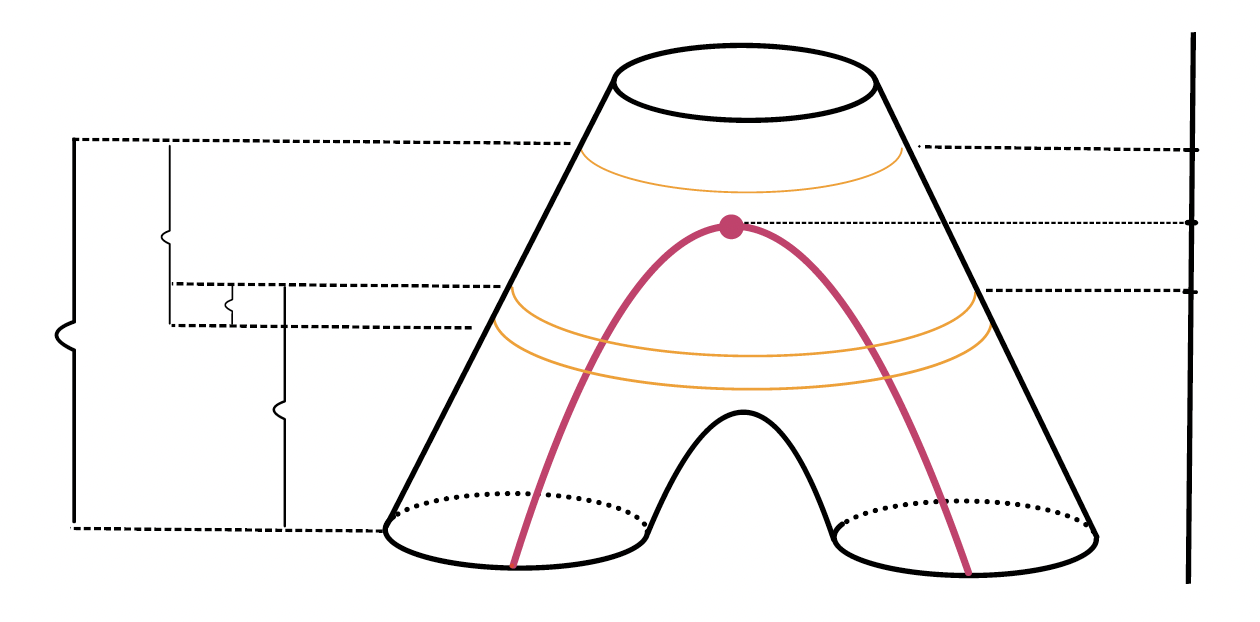}};

\draw (131.7,413.95) node  [font=\normalsize]  {$A$};
\draw (160.7,441.95) node  [font=\normalsize]  {$B$};
\draw (55.7,451.95) node  [font=\normalsize]  {$( M_{I})_{\leq }{}_{v}{}_{+}{}_{\epsilon }{}$};
\draw (155.7,484.95) node  [font=\normalsize]  {$( M_{I})_{\leq }{}_{v}{}_{- \epsilon }{}$};
\draw (611.7,407.95) node  [font=\normalsize]  {$v$};
\draw (621.7,377.95) node  [font=\normalsize]  {$v+\epsilon $};
\draw (621.7,436.95) node  [font=\normalsize]  {$v-\epsilon $};
\draw (394,419.95) node  [font=\footnotesize]  {$p$};
\end{tikzpicture}
}
\caption{Example of Morse data for the point $p$.}
\label{fig:twisted_double}
\end{figure}
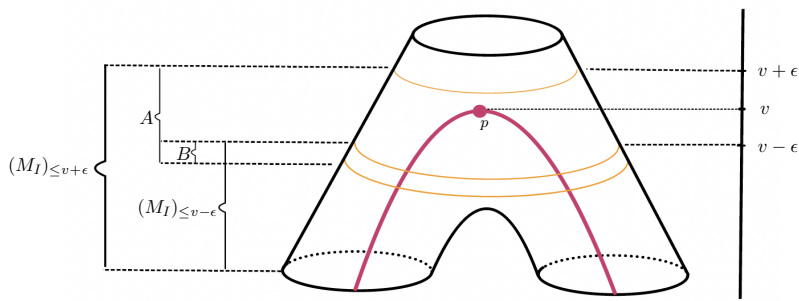

\begin{remark} If the local Morse data from Definition \ref{defn: Morse data} is defined to be $A =(M_{I})_{[v-\epsilon, v+\epsilon]}$ and $B = (M_{I})_{v-\epsilon}$,  \cite{Massey} refers to this as \textit{coarse Morse data.} 
\end{remark}

\begin{lemma}[\cite{SMT88}, I.3.2]\label{lem: Morse no crit values}
    Let $f:M_I\to \bbr$ be a nested Morse function. If an interval $[a,b]$ contains no critical values of $f$, then $(M_I)_{\leq a}$ is nested homeomorphic to $(M_I)_{\leq b}$.
\end{lemma}

 We have defined a nested Morse function as individually Morse with critical points and values being distinct. Since the critical values of a nested Morse function are isolated, it suffices to look at Morse data for a small neighborhood around a critical point $p$; we describe this construction of \emph{local Morse data}.
Provide $M_{d_n}$ with a smooth Riemannian metric. \cite{SMT88} shows that for any critical point $p$ on $F_i$ with critical value $v$ a stratified (and thus a nested) Morse function $f$ has a small neighborhood $B_\delta (p)$ of radius $\delta>0$ such that $\partial B_\delta(p)$ intersects $F_{j>i}$ transversely and such that none of the other critical points of $f$ in $B_\delta(p)$ have critical value $v$.

\begin{definition} 
    Choose $\delta>0$ as above.
    The \emph{local Morse data} of $f$ at $p$ is the pair $$(B_\delta(p) \cap f^{-1}([v-\epsilon,v + \epsilon]),B_\delta(p)\cap f^{-1}(v - \epsilon))$$
\end{definition}
Local Morse data describes how the topology of the level set $$B_\delta(p)\cap f^{-1}(x \in [v - \epsilon, v+ \epsilon])$$ changes as you pass the critical point $p$, in a small neighborhood of $p$.
Theorem 3.5.4 of \cite{SMT88} states that for critical points with isolated critical values, local Morse data is Morse data. 

Further, the local Morse data for a critical point $p$ splits into a tangential and normal component. More specifically, there exists a $\delta>0$ sufficiently small such that $\partial B_\delta(p)$, the boundary of a small neighborhood around $p$, is transverse to each stratum $F_j$. Let $F_i$ be the stratum containing $p$. Let $N'$ be a smooth submanifold of $M_{d_n}$ which is transverse to each stratum of $F$, intersects $F_i$ in the single point $p$, and satisfies $\dim(F_i)+\dim(N')=\dim(M_{d_n})$.
The normal slice $N$ through $F_i$ at $p$ is the set
\[ N:= N'\cap B_\delta(p).\]

\begin{definition}
    Let $p$ be a critical point of $f$ contained in the stratum $F_i$.
    The \emph{tangential Morse data} for $f$ at $p$ is the local Morse data for 
    $f|_{F_i}$ at $p$, and the \emph{normal Morse data} for $f$ at $p$ is the local Morse data for $f|_{N}$ at $p$. 
\end{definition}

\begin{figure}[!h]
    \centering
    \begin{tikzpicture}
        \draw (0,0) -- (2,0);
        \draw (2,0) -- (2,2);
        \draw (2,2) -- (0,2);
        \draw (0,2) -- (0,0);
        \draw[color=blue] (1,0) to node[midway, above right] {\color{blue}$N$} (1,2);
        \draw[thick, color=purple, distance=1.5cm] (2,2) to[out=270,in=270] node[midway,below right] {\color{purple}$M_1$} (0,2); 
        \node at (3,1) {$=$};
        
        \node at (3.75, 1) {\color{blue}\Large$($};
        \draw[color=blue] (4, 0.75) -- (4, 1.25);
        \node at (4.25, 0.75) {\color{blue}$,$};
        \node at (4.5, 0.75) {\color{blue}$\cdot$};
        \node at (4.75, 1) {\color{blue}\Large$)$};
        
        \node at (5.5, 1) {$\times$};

        \node at (6.25, 1) {\color{purple}\Large$($};
        \draw[color=purple, distance=0.75cm] (6.5, 1.25) to[out=270,in=270] (7, 1.25);
        \node at (7.2, 0.75) {\color{purple}$,$};
        \node at (7.5, 0.875) {\color{purple}$\emptyset$};
        \node at (7.75, 1) {\color{purple}\Large$)$};
        
        \node at (4.25, 0.25) {\color{blue}$\substack{\text{normal}\\ \text{Morse data}}$};
        \node at (7.2, 0.25) {\color{purple}$\substack{\text{tangential}\\ \text{Morse data}}$};
        \node at (1, -0.25) {$\substack{ \\ \text{local picture}}$};
    \end{tikzpicture}
    \caption{Example of normal and tangential Morse data for a local picture with a critical point of index $1$ in the submanifold $M_1$ and no critical points on $M_2$; $N$ is the normal slice.}
    \label{fig:morsedata}
\end{figure}
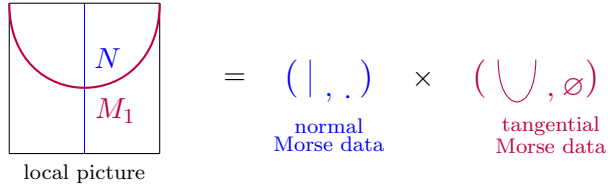

The Main Theorem of Stratified Morse Theory describes local Morse data in terms of tangential and normal Morse data. 

\begin{theorem}[The Main Theorem of Stratified Morse Theory]\label{thm:normaltangentialmorsedata}
    The local Morse data of $f$ at $p$ is homeomorphic to the product of the normal and the tangential Morse data of $f$ at $p$.
\end{theorem}
See Theorem~\ref{thm: SMT main thm}.

\begin{remark}
    In the situation of a nested manifold, the normal and tangential Morse data for a critical point on stratum $F_i$ take a particularly nice form. For Whitney stratified spaces, the topological type of the boundary of the normal slice measures the singularity type of the space along a stratum.  In the case of nested manifolds, this link is a sphere, because the normal Morse data is always a collection of discs of every dimension $d_j - d_i$, for $j > i$.  
\end{remark}

\begin{lemma}\label{lem:unpack Morse data}
  
     For $p$ a critical point on the top dimensional stratum $F_n$, the normal Morse data is a pair consisting of a point and the empty set, $(\bullet, \emptyset)$, and the tangential Morse data is the (unstratified) Morse data of $M_{d_n}$ (and has an empty intersection with $M_{d_i}$ for all $i< n$).

    For $p$ a critical point on $F_i$, for $i<n$, the normal Morse data is the relative pair $(A, B)$ where 
    \begin{itemize}
    \item $A$ is a nested disk $D_{I-d_i} := (D^{d_n-d_i}, D^{d_{n-1}-d_i}, \dots, D^{d_{i+1} -d_i})$ in the (nested) normal bundle $\nu$ of $M_{d_i}$ in $M_{d_{i+1}} \subset M_{d_{i+2}}\subset \dots \subset M_{d_n}$ at $p$, and
    \item $B$ is the lower point of the disk in the Morse function (the point $x \in D_{I-d_i}$ such that $f(x) = v - \epsilon$). 
    \end{itemize}
    The tangential Morse data is the local Morse data of $f_i$ as a Morse function on $M_{d_i}$.  
\end{lemma}
\begin{proof}
For a critical point in the top dimensional stratum $p\in F_n=M_{d_n}\setminus M_{d_{n-1}}$, note that the normal slice a point; thus the normal Morse data is is this point $p$ relative to the empty set. By definition, the tangential Morse data is the local Morse data for $f|_{M_{d_n}\setminus M_{d_{n-1}}}$ at $p$. Note that the neighborhood $B_\delta(p)$ can be chosen small enough such that $B_\delta(p)\cap M_{d_{n-1}}=\emptyset$; thus the tangential Morse data is the Morse data of $f$ at $p$ (considered as a point in $M_{d_n}$, without regard to the stratification). 

For $p \in F_i=M_{d_i} \setminus M_{d_{i-1}}$, $i<n$, first note that the fiber of the normal bundle at $p$ is a nested space consisting of all the spaces that arise from considering the (non-nested) normal bundles of $M_{d_i}$ in the bigger strata $M_{d_j}$, where $d_i < d_j \leq d_n$, that is, $\nu_p := (\nu(M_{d_i} \subset M_{d_{n}})_p, \nu(M_{d_i} \subset M_{d_{n-1}})_p \dots \nu(M_{d_i} \subset M_{d_{i+1}})_p )$.  The normal slice is the intersection of this nested space $\nu_p$ with a small neighborhood $B_{\delta}(p)$,
$$\nu_p \cap B_{\delta}(p) = (D^{d_n-d_i}, D^{d_{n-1}-d_i}, \dots, D^{d_{i+1} -d_i}).$$

Like in the case of the top stratum, $B_{\delta}(p)$ can be chosen small enough so that the intersection with lower dimensional strata is empty.  
\end{proof}

\begin{theorem}\label{thm:nested_critical_points}
    Let $M_I$ be a nested manifold and let $f\colon M_I\to \bbr$ be a nested Morse function. The critical points of $f$ are of the form $p^i_j$, with $0 \leq j \leq d_i$, an index $j$ critical point of $M_{d_i}$.
\end{theorem}

\begin{proof}
    Since $f$ is nested Morse, the critical points of $f_i$ for all $i$ are distinct. 
    By \cref{lem:unpack Morse data}, the possible Morse data for a 
    critical point on the stratum $F_i$ is given by the possible Morse data of $f_i$ as a Morse function on $M_{d_i}$. By the usual arguments, the possible Morse points of $f_i$ are of the form $p^i_j$, with $0 \leq j \leq d_i$, an index $j$ critical point of $M_{d_i}$. 
\end{proof}

In the specific case of a nested surface $M_{1<2}$, the Morse data takes the following forms.

\begin{corollary}
    Let $M_{1<2}$ be a nested surface and let $f\colon M_{1<2}\to \bbr$ be a nested Morse function. The critical points of $f$ are of the following form:
    \begin{itemize}
        \item A critical point $p^2_0$ on $M_2$ with index $0$.
        \item A critical point $p^2_1$ on $M_2$ with index $1$.
        \item A critical point $p^2_2$ on $M_2$ with index $2$.
        \item A critical point $p^1_0$ on $M_1$ with index $0$ and Morse data given by 
        $(\vert,\bullet)\times (\cup, \emptyset)$
        \item A critical point $p^1_1$ on $M_1$ with index $1$ and Morse data given by 
        $(\vert,\bullet)\times (\cap, \bullet \ \bullet)$
    \end{itemize}
\end{corollary}

\begin{figure}[ht]
\adjustbox{scale=0.90}{

\tikzset{every picture/.style={line width=0.95pt}} 

\begin{tikzpicture}[x=0.75pt,y=0.75pt,yscale=-1,xscale=1]

\draw (322.5,52.74) node  {\includegraphics
[width=422.25pt,height=73.1pt]
{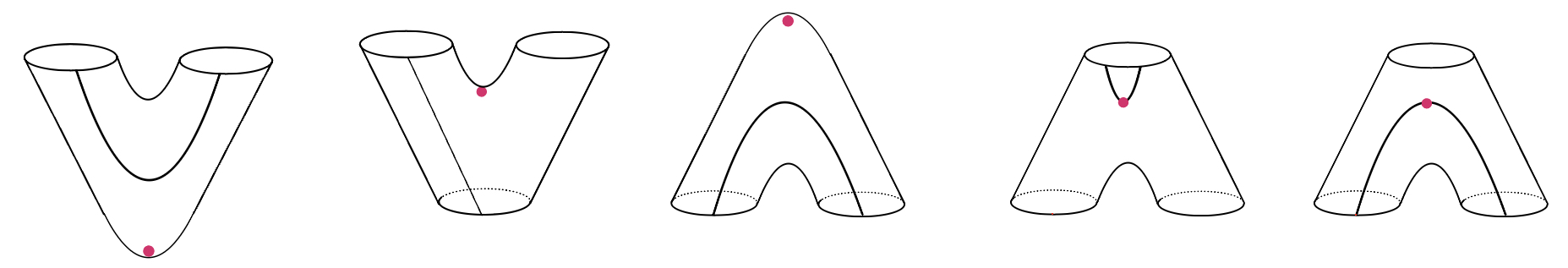}};

\draw (96,82.95) node  [font=\footnotesize]  {$p_{0}^{2}$};
\draw (215,46.95) node  [font=\footnotesize]  {$p_{1}^{2}$};
\draw (325,22.95) node  [font=\footnotesize]  {$p_{2}^{2}$};
\draw (446,50.95) node  [font=\footnotesize]  {$p_{0}^{1}$};
\draw (554,51.95) node  [font=\footnotesize]  {$p_{1}^{1}$};

\end{tikzpicture}

}
\caption{The points in red mark examples of each type of critical point that a nested Morse function $M_{1<2} \to \mathbb{R}$ can have. Note that these pictures do not just show the local neighborhood of the indicated points but also include other critical points of different types.
}
\label{fig:critical-points-index}
\end{figure}
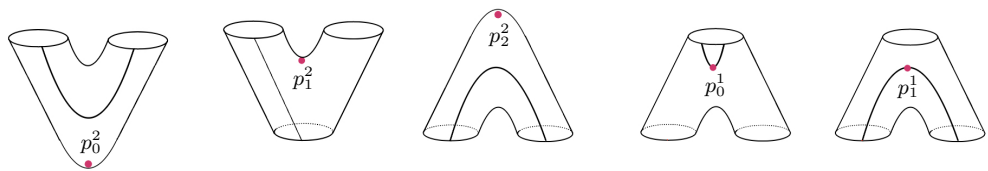

\subsection{Nested Cerf decompositions}
Using the nested Morse theory of the previous section, we now outline how any nested cobordism can be written as a composition of elementary cobordisms (see~\cref{defn: element cob}); these elementary cobordisms will be the generators of our cobordism category.

\begin{definition}[\cite{FreedNotes}, Definition 23.6]\label{defn: excellent}
    Let $W_I\colon M_0 \Rightarrow M_1$ be a nested cobordism. A nested Morse function $f\colon W_I\to \bbr$ is \emph{excellent} if 
    \begin{enumerate}
        \item $f(M_0)=a_0$ is the minimum of $f$;
        \item $f(M_1)=a_1$ is the maximum of $f$.
    \end{enumerate}
    We will call the critical points of $f$ $p_1,\dots,p_N$, with the respective critical values $v_1,\dots, v_N$ which satisfy
        \[ a_0 < v_1<\dots<v_N < a_1.\]
\end{definition}

\begin{lemma}\label{lemma:elementary}
    Given any nested cobordism $W_I\colon M_0 \Rightarrow M_1$, an excellent nested function $f\colon W_I \to \bbr$ always exists. 
\end{lemma}
\begin{proof}
Note that we can find an excellent function $f\colon W_{d_n}\to \RR$ on the top dimensional manifold, see e.g. \cite[Lemma 2.6]{milnor}.
The proof in \cite{SMT88} showing that stratified (and thus nested) Morse functions are dense in the set of all smooth proper functions relies on application of Thom transversality on the map from $W_{d_n}$ into the jet space defined by the function $f$.
By the extension theorem for Thom transversality \cite[Chapter 2.3]{guillemin2010differential}, we can perturb the function $f$ to be transverse while keeping it constant on $\partial W$.
Hence, we can apply Theorem~\ref{thm: nested Morse dense in smooth} to perturb $f$ to a nested Morse function $f'\colon W_I \to \bbr$ while maintaining the condition that it is excellent.
\end{proof}

We use the notion of an excellent nested function to decompose our nested cobordisms into their elementary building blocks, called \emph{elementary cobordisms}.
\begin{definition}\label{defn: element cob}
    A nested cobordism $W_I\colon M_0\Rightarrow M_1$ is an \emph{elementary cobordism} if it admits an excellent nested function with at most one critical point. 
\end{definition}

\begin{lemma}
    Any nested cobordism between $M_{I-1}$ and $M'_{I-1}$ can be decomposed into elementary cobordisms. 
\end{lemma}
\begin{proof}
    Let $W_I$ be a nested cobordism between $M_{I-1}$ and $M'_{I-1}$. By \cref{lemma:elementary}, there is an excellent nested function $f$ on $W_I$. Choose regular values $b_1,\dots, b_{N-1}$ satisfying\[
    a_0< b_1 < v_2 < \dots < b_{N-1} < v_N < a_1.
    \] Write $b_0=a_0$ and $b_N=a_1$. Then for each $1\leq i\leq N$, the nested submanifold $W_i:= f^{-1}([b_{i-1}, b_{i}])$ has at most one critical point, and hence is an elementary cobordism between $f^{-1}(b_{i-1})$ and $f^{-1}(b_{i})$, with $f^{-1}(b_0) = M_{I-1}$ and $f^{-1}(b_N) = M'_{I-1}$. Then the composition\[
    W_{N}\circ \dots \circ W_2\circ W_1
    \] is the claimed decomposition.
\end{proof}

\begin{definition}[\cite{GWW}, Defn 2.3]
    A \emph{Cerf decomposition} of a nested cobordism $W$ is a decomposition into a sequence of elementary cobordisms\[
    W = W_1 \cup_{M_1} \dots \cup_{M_{n-1}} W_n
    \] such that \begin{itemize}
        \item Each $W_i\subseteq W$ is an elementary $I$-nested cobordism embedded in $W$,
        \item Each $M_i\subseteq W$ is an embedded $(I-1)$-nested submanifold of $W$,
        \item The $W_i$ are disjoint from each other in $W$, except that $W_i\cap W_{i+1}\cong M_i$ for $i=1,\dots, n-1$
        \item $W_1\cap  \bndry W = \bndry W^-$ and $W_n\cap \bndry W = \bndry W^+$. 
    \end{itemize}

    Analogously, a \emph{Cerf decomposition} of a morphism $[W]$ in $\Cob_I$
    is a sequence $[W_1],\dots, [W_n]$, where $W_i$ are elementary cobordisms, that compose\[
    [W] = [W_1]\circ \dots \circ [W_n].
    \]
\end{definition}

\begin{lemma}\label{lem:Cerf on cob iff Cerf on class}
    A Cerf decomposition of a cobordism $W_{I}$ induces a Cerf decomposition on its diffeomorphism class $[W_{I}]$. Moreover, every Cerf decomposition of a class $[W_{I}]$ arises from a Cerf decomposition of a representative cobordism.
\end{lemma}
\begin{proof}
If $W = W_1\cup_{M_1} \dots \cup_{M_{n-1}} W_n$ is a Cerf decomposition of the nested cobordism $W=W_I$, then the intersection conditions on the $M_i$ ensure that \[
[W] = [W_1]\circ \dots \circ [W_n]
\] in $\Cob_{I}$. On the other hand, suppose that $[W]=[W_1]\circ \dots \circ [W_n]$ is a Cerf decomposition of the morphism $[W]$ in $\Cob_{I}$. Choose representatives $W_1,\dots, W_n$ for each of the cobordism classes that admit collar neighborhoods of the shared boundaries $M_i$ in both $W_i$ and $W_{i+1}$. 
Then the glued cobordism\[
W' = W_1\cup_{M_1} \dots \cup_{M_{n-1}} W_n
\] is a representative of $[W]$ and $W'$ has a Cerf decomposition via the embeddings $W_i\hookrightarrow W'$, $M_i\hookrightarrow W'$.
\end{proof}

\begin{corollary}
    Any nested cobordism has a Cerf decomposition.
\end{corollary}\begin{proof}
 Let $W$ be a nested cobordism. By \cref{lemma:elementary}, there is an excellent nested Morse function $f\colon W\to \RR$ and regular values $b_0<b_1<\dots <b_n$ such that\[
    W = W_1\cup_{M_1} \dots \cup_{M_{n-1}} W_n
    \] is a Cerf decomposition, where $W_i := f^{-1}([b_{i-1}, b_i])$ are elementary bordisms between the level-sets $M_i:= f^{-1}(b_i)$. Note that the properties of $f$ we need here are:\begin{itemize}
        \item $f^{-1}(b_0) = \bndry W^-$ and $f^{-1}(b_n) = \bndry W^+$,
        \item there is a bijection $\crit(f)\to f(\crit(f))$ between critical points and critical values (i.e. $f$ has distinct values at each of the critical points, which are isolated from each other),
        \item $b_0,\dots, b_n\in \RR$ are regular values of $f$ so that each $(b_{i-1},b_i)$ contains at most one critical value of $f$.
    \end{itemize}
\end{proof}

This corollary, combined with \cref{lem:Cerf on cob iff Cerf on class}, implies the following result.

\begin{corollary}\label{cor: every cob class has Cerf decomp}
    Every morphism in $\Cob_{I}$ has a Cerf decomposition.
\end{corollary}

The elementary cobordisms in our Cerf decomposition have zero or one critical point. \cref{thm:nested_critical_points} gives a complete list of the possible types of critical points. Elementary cobordisms without critical points are given by mapping cylinders.
\begin{lemma}
Elementary cobordisms with zero critical points are mapping cylinders of self-diffeomorphisms up to pseudo-isotopy.
\end{lemma}
\begin{proof}
Let $f_I$ be an excellent nested function with no critical points on a cobordism $W_I$.
By \cref{lem: Morse no crit values}, $f^{-1}(a)$ is nested homeomorphic for every value of $a$ in the image. Hence, $W_I \cong M_{I-1} \times [0,1]$. 
It follows from \cref{cor: automorphisms are diffeos up to pseudo-isotopy} that nested cobordisms of this form are given by mapping cylinders of diffeomorphisms of the boundary up to nested pseudo-isotopy.
\end{proof}

\section{The striped cylinder cobordism category $\cyl$}\label{section3:cyl}

\subsection{Defining $\cyl$}

We now restrict our attention to $\Cob_{1<2}$, which is the nested cobordism category with objects $(0<1)$-manifolds and morphisms diffeomorphism classes of $(1<2)$-cobordisms between them. 
In the current paper we consider the subcategory $\cyl^c$ of $\Cob_{1<2}$ where objects are $(0<1)$-manifolds given by points on $S^1$ and cobordisms are restricted to nested surfaces where the surface is $S^1\times [0,1]$.  
We will moreover quotient this cobordism subcategory by the relation that contractible circles are set to zero. The resulting category we denote $\cyl$. 

\begin{definition}
    Let $M=(S^1,M_0)$ and $M'=(S^1, M'_0)$ be \textit{marked circles}: closed, oriented $(0<1)$-manifolds with background manifold $S^1$. A \textit{
    striped cylinder cobordism} from $M$ to $M'$ is an oriented $(1<2)$-manifold with boundary, $C=(C_2, C_1)$, where $C_2=S^1\times [0,1]$, 
    along with orientation-preserving nested diffeomorphisms 
    \[\begin{tikzcd}
            M \ar[r, hook] & C & \overline{M'}\ar[l, hook']
    \end{tikzcd}
    \] which map $M$ and $M'$ diffeomorphically (as nested manifolds) onto the in- and out-boundary of $C$, respectively.
\end{definition}

\begin{definition}
    $\cyl^c$ is the subcategory of $\Cob_{1<2}$ with objects marked circles and morphisms striped cylinder cobordisms up to nested diffeomorphism equivalence.
\end{definition}

\begin{definition}\label{defn:equiv reln on nested cobs}
Let $W$ be any 
nested $(1<2)$-cobordism.
We have $W_1 = W_1^\bndry  \sqcup W_1^{nc} \sqcup W_1^c$, where $W_1^\bndry$ are the components of $W_1$ with boundary, $W_1^{nc}$ are components that map non-trivially into $\pi_1 (W_2)$ and $W_1^{c}$ are contractible loops in $W_2$. We define the \textit{circle reduced} version of $W$ to be $\widetilde{W}=(W_2, W_1^\bndry  \sqcup W_1^{nc})$.
Two nested cobordisms $W$ and $W'$ from $M$ to $M'$ are called \textit{circle equivalent} if we have a diagram
\[
    \begin{tikzcd}
        & W \ar[dd, "f"] &\\
       M\ar[ru, hook]\ar[rd, hook] && \overline{M'}\ar[ul, hook'] \ar[ld, hook']\\
        & W' &
    \end{tikzcd} \] 
such that $f$ is a nested diffeomorphism when restricted to $\widetilde{W}$ and $\widetilde{W'}$. 
\end{definition}

Note that the quotient map taking nested cobordisms to their circle equivalence class is well-defined on diffeomorphism classes of cobordisms and leaves the in- and outgoing boundaries of the cobordism invariant, so that we can make the following definition.

\begin{definition}
Let $\Cob_{1<2}^r$ be the \textit{circle reduced nested cobordism category} with morphisms given by nested cobordisms modulo diffeomorphism and circle equivalence, and let $F: \Cob_{1<2} \rightarrow \Cob_{1<2}^r$ be the canonical quotient functor.
\end{definition}

The definition below will be useful in \cref{section5:TLalgebras}.

\begin{definition}
Let $\Cob_{1<2}^a$ be the category described as follows:
\begin{itemize}
        \item The objects of $\Cob_{1<2}^a$ are those of $\Cob^r_{1<2}$.
        \item A morphism $\alpha\colon S^1_n \to S^1_m$ is an equivalence class of $(1<2)$-nested cobordisms in $\Cob^r_{1<2}$ along with a natural number $\mu\in \mathbb{Z}_{\geq 0}$. 
        \item  Let $\mu(\alpha, \beta)$ denote the number of new contractible loops that is formed by the composition $\alpha\circ \beta$ of two nested cobordisms $\alpha$ and $\beta$.
        Composition in $\Cob_{1<2}^a$ are given by $(\alpha, \mu)\circ (\beta, \nu) = (\alpha \circ \beta, \mu+\nu+\mu(\alpha, \beta))$, where $\alpha \circ \beta$ is composition of nested cobordism classes as in $\Cob^r_{1<2}$ (with the $\mu(\alpha, \beta)$-many contractible closed loops removed).
    \end{itemize}
\end{definition}

\begin{remark}
The functor $F$ factors as  
        \[\begin{tikzcd}
       \Cob_{1<2} \ar[r, "F^a"] \ar[rr, bend right, swap, "F"]
       & \Cob_{1<2}^a \ar[r, "F^r"] & \Cob_{1<2}^r  
    \end{tikzcd}
    \] 
\end{remark}

\begin{definition}\label{defn:cyla and cyl}
The categories $\cyl^a$ and $\cyl$ are defined as the image in $\Cob_{1<2}^a$ and $\Cob_{1<2}^r$ of the functors $F^a$ and $F$ respectively, restricted to the subcategory $\cyl^c$.
        \[\begin{tikzcd}
       \Cob_{1<2} \ar[r, "F^a"] &  \Cob_{1<2}^a \ar[r, "F^r"] & \Cob_{1<2}^r \\
        \cyl^c  \ar[u, hook] \ar[r, swap, "F^a"] & \cyl^a  \ar[u, hook] \ar[r, swap, "F^r"] & \cyl\ar[u, hook]
    \end{tikzcd}
    \] 
\end{definition}

We will now restrict ourselves to considering the category $\cyl$.
As in~\cref{thm: pushouts give composition}, composition in $\cyl$ is again given by pushouts that are defined up to diffeomorphism. 

Such a composition may create new contractible circles, in which case the composite is circle equivalent to the cobordism with these circles removed.
Immediate from the definition is the fact that composition is associative and that cylinders that are nested diffeomorphic to $(C_2, C_1) = (S^1\times[0,1], M_0\times [0,1])$ (potentially decorated with contractible circles), together with inclusion maps on the boundary that are pseudo-isotopic to the identity, are identity morphisms in the category.

\begin{remark}
    Up to nested diffeomorphism, oriented $(0<1)$-manifolds with background manifold diffeomorphic to $S^1$ are given by $S^1$ with a certain number of marked points. 
    By \cref{propn: diffeo gives cobord}, nested diffeomorphic manifolds are isomorphic as objects in the category. Since a category is equivalent to its skeleton, we can think of $\cyl$ as having objects given by diffeomorphism classes of one circle with $k$ marked points for every $k\geq 0$, which we denote $S_k^1$. In order to keep track of the way we compose cobordisms, we endow $S_k^1$ with a preferred marked point which we denote by $0$. We orient $S^1_k$ clockwise and label the other marked points $1,\dots, k-1$ accordingly. 
\end{remark}

\begin{remark}\label{rmk:Cob 2cat}
    Note that $\Cob_{1<2}$ could also be thought of in the context of a fully extended 2-dimensional cobordism category, but where there is the additional data of the nested structure (see, for example, \cite{schommerpries2014classification} or \cite{LP:extended} for the non-nested case). In this case, $\Cob_{1<2}$ could be described as the hom-category arising from endomorphisms of the object $\varnothing$. It would be an interesting question to explore the algebraic structure this nested version of the fully extended cobordism category would give, but outside the scope of the current work.
\end{remark}

\subsection{Generators for $\cyl$}

In the case of a morphism $C_{1<2}$ in $\cyl$, since there are no critical points on $C_2$, the elementary cobordisms only involve critical points on the 1-dimensional submanifold $C_1$. The following definitions are similar to ones in~\cite[Section 2.2]{penneys}, but in our case there are no shadings. 

\begin{definition}\label{defn: defn of gen}
We introduce the following names for these elementary cobordisms in $\cyl$:
    \begin{itemize}
        \item[$\id_k$:] The identity cobordism on $S^1_k$
        \item[$\tw_k$:] The twist on $S^1_k$, in the clockwise direction; meaning that point $i$ is connected to point $i+ 1 \pmod{k}$.
        \item[$\bi_k^i$:] The birth cylinder cobordism that maps $S^1_k$ to $S^1_{k+2}$, where the birth arc goes from point $i$ to point $i+1\pmod{k+2}$  on $S^1_{k+2}$, and is isotopic to the clockwise arc from point $i$ to point $i+1\pmod{k+2}$ on $S^1_{k+2}$. 
        The points on $S^1_k$ are connected to the remaining points on $S^1_{k+2}$ by an arc as follows:
            \begin{enumerate}
                \item if $i=0$, point $0$ is connected to point $2$;
                \item if $0<i<k+1$, point $0$ is connected to point $0$;
                \item if $i=k+1$, point $0$ is connected to point $k$.
            \end{enumerate}
            This assignment determines how the remaining points are attached.
        
        \item[$\de_k^i$:] The death cylinder cobordism that maps $S^1_{k}$ to $S^1_{k-2}$, where the death arc goes from point $i$ to point $i+1\pmod{k}$ on $S^1_{k}$, and is isotopic to the clockwise arc from point $i$ to point $i+1\pmod{k}$ on $S^1_{k}$.
        The remaining points on $S^1_k$ are attached to the points on $S^1_{k-2}$ as follows:
            \begin{enumerate}
                \item if $i=0$, point $2$ is connected to point $0$;
                \item if $0<i<k-1$, point $0$ is connected to point $0$;
                \item if $i=k-1$, point $k-2$ gets attached to point $0$.
            \end{enumerate}
            This assignment determines how the remaining points are attached.
    \end{itemize}
\end{definition}
\begin{figure}[h!]
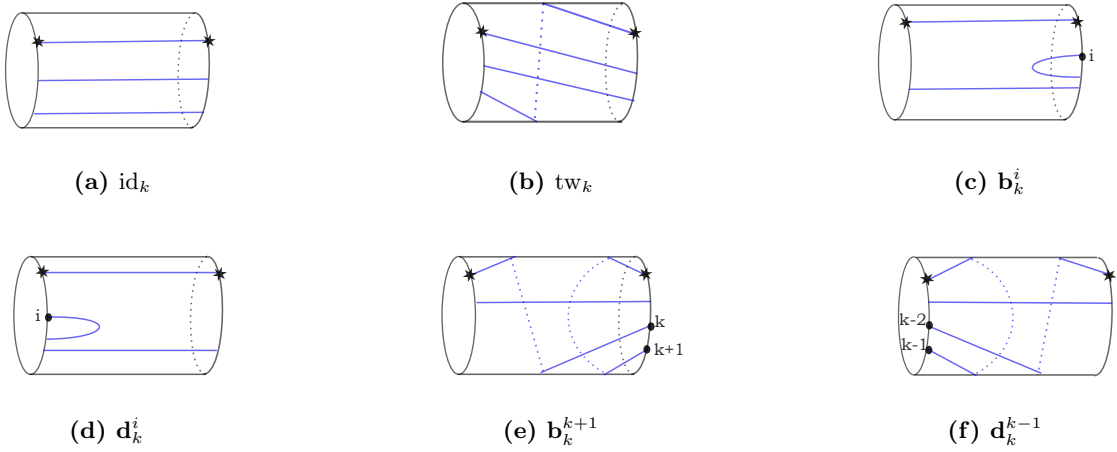

     \centering
     \begin{subfigure}[b]{0.25\textwidth}
         \centering
         \includegraphics[width=\textwidth, trim={2.8cm 22.5cm 11.8cm 3.8cm},clip]{Generators.pdf}
         \caption{$\id_k$}
         \label{fig: identity}
     \end{subfigure}
     \hfill
     \begin{subfigure}[b]{0.25\textwidth}
         \centering
         \includegraphics[width=\textwidth, trim={11.2cm 22.5cm 3.6cm 3.8cm},clip]{Generators.pdf}
         \caption{$\tw_k$}
         \label{fig: twist}
     \end{subfigure}
     \hfill
     \begin{subfigure}[b]{0.25\textwidth}
         \centering
         \includegraphics[width=\textwidth, trim={2.8cm 15.2cm 11.8cm 10.8cm},clip]{Generators.pdf}
         \caption{$\bi_k^i$}
         \label{fig: birth}
     \end{subfigure}
     \hfill
     \begin{subfigure}[b]{0.24\textwidth}
         \centering
         \includegraphics[width=\textwidth, trim={11cm 15.4cm 4cm 10cm},clip]{Generators.pdf}
         \caption{$\de_k^i$}
         \label{fig: death}
     \end{subfigure}
     \hfill
     \begin{subfigure}[b]{0.24\textwidth}
         \centering
         \includegraphics[width=\textwidth, trim={3cm 10.2cm 12cm 15.4cm},clip]{Generators.pdf}
         \caption{$\bi_k^{k+1}$}
         \label{fig: birthedge}
     \end{subfigure}
     \hfill
     \begin{subfigure}[b]{0.24\textwidth}
         \centering
         \includegraphics[width=\textwidth, trim={11cm 10.2cm 4cm 15.4cm},clip]{Generators.pdf}
         \caption{$\de_k^{k-1}$}
         \label{fig: deathedge}
     \end{subfigure}
        \caption{Generating cobordisms}
	\label{relations1}
\end{figure}

\begin{remark}
    Recall that the equivalence classes identify cobordisms that are diffeomorphism equivalent. Thus the definitions of $\bi_k^i, \de_k^i$ are well-defined, as other ways of attaching the remaining points in the prescribed fashion would differ by a Dehn twist. 
\end{remark}

\begin{theorem}\label{thm: gens of cyl}
    The elementary cobordisms in \cref{defn: defn of gen} generate all morphisms in $\cyl$.
\end{theorem}
\begin{proof}
    By~\cref{cor: every cob class has Cerf decomp}, every morphism in $\cyl^c$ can be written as a composition of elementary cobordisms. So it suffices to show that the list from \cref{defn: defn of gen} generates all elementary cobordisms in $\cyl^c$, which will then also provide a complete list of generators for the quotient category $\cyl$.
    First consider elementary cobordisms $C$ with no Morse points. These will be: 
    
    \begin{minipage}{0.95\linewidth}
    \begin{itemize}
        \item[${\id_k:}$] The identity cobordism on $S^1_k$;
        \item[${(\tw_k)^n:}$] Compositions of the positive twist on $S^1_k$, for $1<n<k-1$. \\From here on, we will denote these $\tw_k^n$.
    \end{itemize}
    \end{minipage}
    
    These are all mapping cylinders of pseudo-isotopy classes of diffeomorphisms of $S^1_k$,
    the cobordisms that permute the marked points, giving all the elementary cobordisms without Morse points. Note that the only allowable permutations of the points are by rotation because the the submanifold $C_1$ needs to be embedded.
    Further, $\tw_k^i$ and $\tw_k^{i+k}$ are diffeomorphism equivalent morphisms, by performing a Dehn twist on the cylinder. Thus $\tw_k$ generates both clockwise and counterclockwise twists.  
    
    The elementary cobordisms in $\cyl$ with one Morse point are those where the submanifold $C_1$ has a critical point. The $\bi_k^i, \de_k^i$ account for the Morse point on $C_1$. The other possibilities for how $C_1$ connects the remaining marked points on the circles are given by composing $\bi_k^i,\de_k^i$ with $\tw_k^n$, for various $0<n<k-1$. 
\end{proof}

\begin{remark}
    Note that this list of generating cobordisms for $\cyl$ is not a minimal list. In particular the $\bi_k^i, \de_k^i$ can all be generated from only $\bi_k^0$ and $\de_k^0$ by pre- and post-composing with various degrees of the twist cobordism $\tw_k$.

    We use this extended list of generators in order to write a general nested cobordism in a more efficient normal form, as done in \cref{thm: unique decomp}.
\end{remark}

\subsection{Relations in $\cyl$}\label{section4:relations}
We deduce the following list of relations; this is similar to~\cite[Theorem 2.20]{penneys}.

\begin{theorem}\label{thm: minimal list of generators}
The following relations hold in $\cyl$:

\begin{itemize}
    \item Relations with birth and death in succession: for $k\geq 0$ and $0\leq i,j\leq k$
    \begin{enumerate}
        \item contractible circles:
            $\de_{k+2}^i\circ \bi_k^i=\id_k$,
        \item snake:
        $\de_{k+2}^i\circ \bi_k^j=\id_k$ if $i=j\pm 1$,
    \item no `interaction' between birth and death: \[
    \de_{k+2}^i\circ \bi_{k}^j= \left\{\begin{array}{ll}
       \bi_{k-2}^{j-2}\circ \de_{k}^{i}  &  i< j -1,\\
        \bi_{k-2}^{j}\circ \de_{k}^{i-2} & i > j+1
    \end{array}\right.
    \]
    \end{enumerate}
    \item Relation with births only: for $k \geq 0$ and $0 \leq i, j \leq k$:
    \begin{enumerate}
        \item[(4)] 
    $\bi_{k+2}^i\circ\bi_{k}^j=\bi_{k+2}^{j+2}\circ \bi_k^{i}$ if $i \leq j$,
    \end{enumerate}
    \item Relation with deaths only: for $k \geq 4$ and $0 \leq i, j < k-1$:
    \begin{enumerate}
        \item[(5)] $\de_{k-2}^i\circ\de_{k}^j = \de_{k-2}^{j-2}\circ \de_{k}^{i} $ if $ i < j -1 $,   
    \end{enumerate}
    \item Relations with the twist: For $k\geq 0$
    \begin{enumerate}
        \item[(6)] $\tw_{k+2}\circ \bi_k^i= \bi_k^{i+1}\circ \tw_k$ for $0\leq i \leq k$,
        \item[(7)] $\tw_{k-2}\circ \de_{k}^i = \de_{k}^{i+1}\circ \tw_{k}$ for $0\leq i <  k-1$,
        \item[(8)] $\tw_k^k=\id_k$.
    \end{enumerate}
\end{itemize}
\end{theorem}

\begin{proof}
Relation (1) creates contractible circles that we impose to be the identity (\cref{fig: circle}). Relations (2) through (5) are clear by \cref{relations1} and \cref{relations2}.
Relation (6) is true by picture for $i<k$  (\cref{fig: bt commute}) and we defined $\bi^{k+1}_k$ (\cref{fig: birthedge}) so that relation (6) holds for $i=k$. Relation (7) is also true by picture for $i<k-2$ (\cref{fig: dt commute}) and we defined $\de^{k-1}_k$  (\cref{fig: deathedge}) to make relation (7) hold for $i=k-2$. Relation (8) is true because Dehn twists are diffeomorphic relative boundary to the identity (\cref{fig: dehn}).

\end{proof}

\begin{figure}[h!]
     \centering
     \begin{subfigure}[b]{0.36\textwidth}
         \centering
         \includegraphics[width=\textwidth, trim={5cm 24cm 4cm 0.8cm},clip]{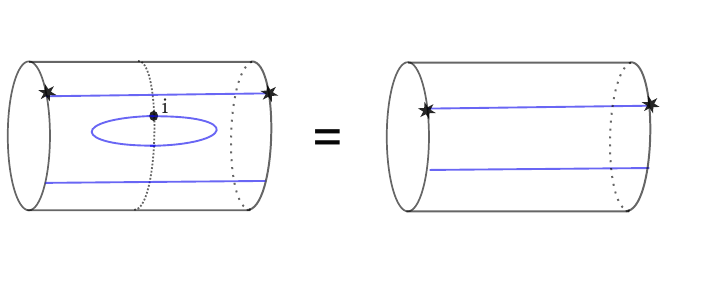}
         \caption{contractible circles}
         \label{fig: circle}
     \end{subfigure}
     \hfill
     \centering
     \begin{subfigure}[b]{0.62\textwidth}
         \centering
         \includegraphics[width=\textwidth, trim={0.5cm 19.8cm 0.5cm 6cm},clip]{Relations_Cylinder_1.pdf}
         \caption{snake}
         \label{fig: snake}
     \end{subfigure}
     \hfill
     \centering
     \begin{subfigure}[b]{0.8\textwidth}
         \centering
         \includegraphics[width=\textwidth, trim={0cm 15cm 0cm 10cm},clip]{Relations_Cylinder_1.pdf}
         \caption{no `interaction' between birth and death}
         \label{fig: bdcommute}
     \end{subfigure}
     \hfill
        \caption{Relations involving interactions between birth and deaths}
	\label{relations1}
\end{figure}

\begin{figure}[h!]
     \centering
     \begin{subfigure}[b]{0.44\textwidth}
         \centering
         \includegraphics[width=\textwidth, trim={3.7cm 24cm 3.8cm 1.2cm},clip]{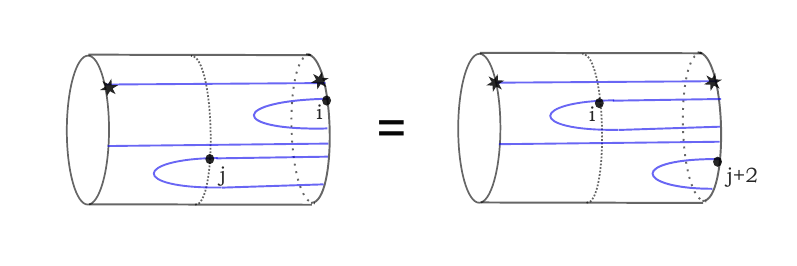}
         \caption{Births commute}
         \label{fig: births commute}
     \end{subfigure}
     \hfill
     \centering
     \begin{subfigure}[b]{0.44\textwidth}
         \centering
         \includegraphics[width=\textwidth,trim={3.7cm 19.6cm 3.8cm 5.6cm},clip]{Relations_Cylinder_2.pdf}
         \caption{Deaths commute}
         \label{fig: deaths commute}
     \end{subfigure}
     \hfill
        \caption{Relations involving birth or deaths moving past each other}
	\label{relations2}
\end{figure}

\begin{figure}[h!]
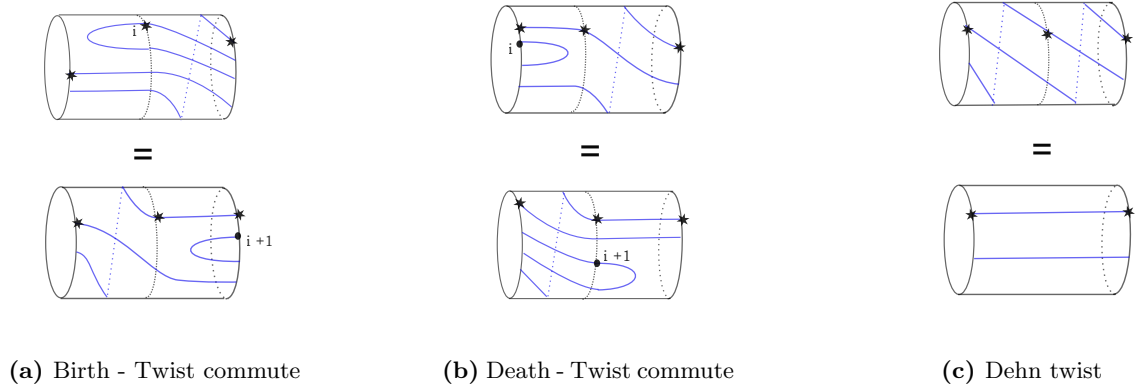

     \centering
     \begin{subfigure}[b]{0.25\textwidth}
         \centering
         \includegraphics[width=\textwidth, trim={1cm 8cm 13cm 13cm},clip]{Relations_Cylinder_2.pdf}
         \caption{Birth - Twist commute}
         \label{fig: bt commute}
     \end{subfigure}
     \hfill
     \centering
     \begin{subfigure}[b]{0.25\textwidth}
         \centering
         \includegraphics[width=0.8\textwidth, trim={8.5cm 8cm 7cm 13cm},clip]{Relations_Cylinder_2.pdf}
         \caption{ Death - Twist commute}
         \label{fig: dt commute}
     \end{subfigure}
     \hfill
     \centering
     \begin{subfigure}[b]{0.25\textwidth}
         \centering
         \includegraphics[width=\textwidth, trim={14cm 7.8cm 0cm 13.2cm},clip]{Relations_Cylinder_2.pdf}
         \caption{Dehn twist}
         \label{fig: dehn}
     \end{subfigure}
     \hfill
        \caption{Relations involving twists}
	\label{relations3}
\end{figure}

\begin{remark}
    The deaths (and births) that cannot be moved past each other are `stacked' (for example: $\de^{j-1}_{k-2} \circ \de^{j}_{k}$ for $1\leq j < k-2$).
\end{remark}

We will show that the relations in \ref{thm: minimal list of generators} are sufficient, but our argument proceeds by putting every cobordism in a normal form. That process is easier to describe by knowing the full set of pairs of births and/or deaths can be moved past each other, which can include $\bi^{k+1}_k$ and $\de^{k-1}_k$. The following corollary will give a complete list that will be used to prove our normal form. 

\begin{corollary}\label{cor: relns in cyl}
The following ``edge case" relations hold in $\cyl$:
\begin{itemize}

    \item Relations where births and deaths interact:
    \begin{enumerate}
        \item[$(0^\ast)\;$] bracelet:  $\de_2^1\circ \bi_0^0=\de_2^0\circ \bi_0^1$
        \item[$(2^\ast)\;$] untwisted snakes: $\de_{k+2}^{k}\circ \bi_k^{k+1} = \id_k$ \hspace{0.2cm} and \hspace{0.2cm} $\de^{k+1}_{k+2} \circ \bi^{k}_k = \id_k$     
        \item[$(2^{\ast\ast})$] twisted snakes: \hspace{0.3cm} $\de_{k+2}^{k+1}\circ \bi_k^0 = \tw_k^2$ \hspace{0.4cm} and \hspace{0.2cm} $\de^0_{k+2} \circ \bi^{k+1}_k=\tw_k^{k-2}$
        \item[$(3^\ast)\;$] no `interaction' between birth and death: 
        \begin{itemize}
            \item[] $\de^{i}_{k+2} \circ \bi^{k+1}_k = \bi^{k-1}_{k-2} \circ \de^i_k$ for $1\leq i\leq k-1$
            \item[] $\de^{k+1}_{k+2} \circ \bi^{i}_k = \bi_{k-2}^{i}\circ \de_{k}^{k-1}$ for $1\leq i \leq k-1$
    
        \end{itemize}
    
    \end{enumerate}
    \item Relation with births only:
    \begin{enumerate}
        \item[$(4^\ast)\;$]  $\bi_{k+2}^i\circ\bi_{k}^{k+1}=\bi_{k+2}^{k+3}\circ \bi_k^{i}$ for $1 \leq i \leq k+1$
    \end{enumerate}
    \item Relation with deaths only: for $k\geq 4$,
    \begin{enumerate}
        \item[$(5^\ast)\;$] $\de^{i}_{k-2} \circ \de_{k}^{k-1} = \de^{k-3}_{k-2} \circ \de^i_{k}$ for $1 \leq i \leq k-3$
    \end{enumerate}
\end{itemize}
\end{corollary}
\begin{proof}
    Each of the relations can be obtained from the relations in~\cref{thm: minimal list of generators} by conjugating by twists to move the points involved in the relation away from the 0 point. The bracelet relation is obtained from relations (6)--(8) and is illustrated in \cref{fig: bracelet}.
\end{proof}
\begin{figure}
    \centering
    \includegraphics[width=0.45\textwidth, trim={4cm 9.5cm 4cm 16.5cm}, clip]{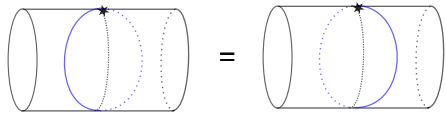}
    \caption{bracelet relation}
    \label{fig: bracelet}
\end{figure}

\begin{remark}
    It is helpful to summarize the following properties of the generators:
\begin{itemize}

    \item We can move births past each other as long as they are not `stacked.'  The `stacked' births are 
        \begin{itemize}
            \item[] $\bi^{j+1}_{k+2} \circ \bi^{j}_k $ for $0 \leq j < k+1$,
            \item[] $\bi_{k+2}^i\circ\bi_{k}^j$ if $(i,j)= (0, k+1), (k+3, 0)$ or $(k, k+1)$.
    \end{itemize} 
    
    \item We can move deaths past each other as long as they are not `stacked.' The `stacked' deaths are 
        \begin{itemize}
            \item[] $\de^{j-1}_{k-2} \circ \de^{j}_{k} $ for $1\leq j < k-2$,
            \item[] $\de^i_{k-2}\circ \de^j_k$ if $(i,j)= (0, k-1), (k-3, 0)$ or $(k-3, k-2)$.
        \end{itemize} 

    \item We can move deaths before births except when they are at the same spot (creating contractible circles that we impose to be the identity) or are adjacent. 
    
    \item When births and deaths are adjacent, they either create `snakes' that cancel the birth and the death (but may add twists) or create `bracelets' that cannot be removed.
    
    \item Births before a twist can always be moved after the twist (likewise with deaths).
    
    \item Births after a twist can always be moved before the twist (likewise with deaths).
\end{itemize}
\end{remark}

\cref{thm: minimal list of generators} and~\cref{cor: relns in cyl} give necessary relations in $\cyl$, and the remainder of this subsection is dedicated to proving that this list of relations is sufficient. 
We will show that any morphism in $\cyl$ has a unique factorization as  some (composition of) $\de_k^i$'s, followed by some (composition of) of twists or bracelets, followed by some (composition of) $\bi_k^i$'s. 

To describe this factorization, we introduce some invariants of nested cylindrical cobordisms, inspired by ~\cite[Definitions 2.11--2.14]{penneys}. For the following definitions, let $C\colon S^1_n\to S^1_m$ be a cylindrical nested cobordism, with $C_1$ the 1-dimensional submanifold of $C_2$. 

\begin{definition}\label{defn: death index}
    Let $S$ be a connected component of $C_1$ such that $|S\cap S^1_n|=2$; i.e. both endpoints of $S$ lie on the ingoing boundary $S^1_n$. Call the collection of all such $S$ the \emph{caps of $C_{1<2}$}. The boundary of $S$ divides $S^1_n$ into two intervals, one of which, $I$, is such that gluing it to $S$ forms a loop that is null-homotopic in the cylinder. Orienting $S^1_n$ clockwise, the first point on $I$ is called the starting point of the cap $S$.
    
    Define $\dind(C)$, the death index of $C$, to be the cyclically ordered sequence of starting points of the caps of $C_{1<2}$. If $C$ has no caps, define $\dind(C)=\emptyset$.
\end{definition}

\begin{definition}\label{defn: birth index}
      Let $S$ be a connected component of $C_1$ such that $|S\cap S^1_m|=2$; i.e. both endpoints of $S$ lie on the outgoing boundary $S^1_m$. Call the collection of all such $S$ the \emph{cups of $C_{1<2}$}. Analogous to the caps, orienting the $S^1_m$ clockwise defines the starting point of $S$.
    
    Define $\bind(C)$, the birth index of $C$, to be the cyclically ordered sequence of starting points of the cups of $C_{1<2}$.
    If $C$ has no cups, define $\bind(C)=\emptyset$.
\end{definition}

\begin{definition}\label{defn: total twist}
    A \emph{through string} of $C\colon S^1_n\to S^1_m$ is a connected component $S$ of $C_1$ where $S\cap S^1_n\neq \emptyset$ and $S\cap S^1_m\neq \emptyset$. The set of all through strings of $S$ is denoted $\rm{ts}(C)$ and we define $\tau(C):=|\rm{ts}(C)|$.

    Starting from the marked point on the incoming circle, number the points connected to through strings by $1,\dots, \tau(C)$; similarly number the points on the outgoing circle which are connected to through strings. Define $1\leq t_0(C)\leq \tau(C)$ so that the first through string connects $1$ and $t_0(C)$.
\end{definition}
The number $t_0(C)$ captures the amount of `twist' that the through strings undergo, ignoring the locations of the births and deaths.
If $\TT(C)=0$, then the morphism $[C]$ factors over $S^1_0$ as
\[C\colon S^1_n\to S^1_0\to S^1_0\to S^1_m.\]
In this case $C$ may have non-contractible loops or \emph{bracelets} in the center cobordism.

\begin{definition}\label{defn: bracelet}
    Define the bracelet number $\br(C)$ to be the number of non-contractible loops in $C$.
\end{definition} 

\begin{remark}
    Note that only one of $\TT(C)$ and $\br(C)$ can be non-zero.
\end{remark}

\begin{lemma}
The invariants $\bind, \dind, \TT, t_0, \br$ descend to well-defined invariants of the morphisms of $\cyl$. 
\end{lemma}
\begin{proof}
    It suffices to show that the invariants are preserved under orientation-preserving diffeomorphisms that fix the boundary, and circle equivalence, since these are the relations used to define the morphisms in $\cyl$ (\cref{defn:equiv reln on nested cobs}).
    The invariants $\bind, \dind$ are determined by the order in which the connected components of $C_1$ intersect the boundary, which is preserved by orientation-preserving diffeomorphisms that fix the boundary pointwise.     

The number of bracelets $\br(C)$ is preserved since nested diffeomorphisms preserve non-contractible loops. Similarly, the number of through strings (and their sources and targets) is preserved under nested diffeomorphism, so $\tau(C)$ and $t_0(C)$ are also preserved.
None of the invariants depend on the presence of contractible loops.
\end{proof}

This lemma shows that if $C$ and $C'$ are two representatives of the same morphism in $\cyl$, then they have the same invariants. The remainder of this section is dedicated to proving the converse.
The above invariants determine specific ``types'' of nested cylindrical cobordisms (see the analogous description in~\cite[Definition 2.24]{penneys}), which we will use to give a normal form for any nested cylindrical cobordism, just as in \cite[Theorem 2.38]{penneys}. Our methods are analogous to those of \cite[Section 2.4]{penneys} in the shaded case.

\begin{definition}
    A cylindrical cobordism $C$ is of 
    
    \begin{minipage}{0.95\linewidth}
    \begin{itemize}
        \item[Type I:]  if it is a composition of only deaths (i.e. $C\colon S^1_n\to S^1_m$ where $n>m$ and $C=\de_{m+2}^k \circ \de_{m+4}^j \circ \cdots \de_n^i$); or the identity cobordism;
        \item[Type II:]  if it is a composition of bracelets; or a composition of twists ($C=\tw_n^i\colon S^1_n\to S^1_n$); or the identity cobordism;
        \item[Type III:] if it is a composition of only births (i.e. $C\colon S^1_n\to S^1_m$ where $n<m$ and $C=\bi_{m-2}^k \circ \bi_{m-4}^j \circ\cdots\bi_n^i$); or the identity cobordism.
    \end{itemize}   
    \end{minipage}
\end{definition}

Note that if $C$ is Type I then 
$\br(C)=0$ and $\bind=\varnothing$; if $C$ is Type II then $\bind=\dind=\varnothing$; if $C$ is Type III then 
$\br(C)=0$ and $\dind=\varnothing$. From the definition of types, it is clear that the types are closed under compositions.

\begin{lemma}\label{lem: types uniquely det by invariants}
If cylindrical cobordisms $C$ and $C'$ are of Type I with $\dind(C)=\dind(C')$ and $\tau(C) = \tau(C')$, then $C$ and $C'$ are connected by a finite sequence of relations from~\cref{thm: minimal list of generators}, and hence $[C]=[C']$ in $\cyl$.
The analogous statement holds for Type III using the birth index. 
\end{lemma}
\begin{proof}
    We will prove the lemma assuming $C$ is of Type I; if $C$ is of Type III, a ``dual'' argument can be used. First observe that we must have $C\colon S^1_n\to S^1_m$ for $m=\tau(C)$ and $n=\tau(C)+2\lvert{\dind(C)\rvert}$. It thus suffices to show that the combinatorics of how the death arcs are stacked is uniquely determined by $\dind(C)$, at which point any two representatives are connected via relations (6) and (6*) in~\cref{thm: minimal list of generators}. The claim follows by a straightforward combinatorial argument that $\dind(C)$ determines not only the data of the starting points of the death arcs, but also the end points.
\end{proof}

\begin{lemma}\label{lem: type II from invariants}
  If cylindrical cobordisms $C, C'$ are of Type II with $\TT(C)=\TT(C')$ and $\br(C)=\br(C')$, then $C$ and $C'$ are connected by a finite sequence of relations from~\cref{thm: minimal list of generators}, and hence $[C]=[C']$ in $\cyl$.
\end{lemma}
\begin{proof}
        Assume $C$ is of Type II. If $\br(C)\neq 0$, then we want to show $C$ is the composition of $\br(C)$-many bracelet cobordisms: $\de_{2}^{1}\circ \bi_0^0 \colon S^1_0\to S^1_0$. 

        Whenever a $\de_k^i$ occurs in $C$, move it as far to the beginning of the cobordism (to the left in the picture; to the right in the factorization) as possible. Because there is a non-contractible loop, there will be at least one $\bi^i$ to the left of the picture and at least one $\de^i$ to the right of the picture. As we move deaths earlier, the following may happen: the death will cancel with a birth (relations (1), (3), (3*)); or will be unable to move past a birth (as in the bracelet relation (2) or the contractible circle relation (1)); or will move past a birth or twist (relations (4), (4*)). This could lead to the death at the beginning of the cobordism in general, but in this case, $\dind(C)=\emptyset$, so only the first two options are possible. This gives a representation of $[C]$ where the deaths are immediately after a birth that they cannot move past, which only happens in the bracelet cobordism $\de_{2}^{1}\circ \bi_0^0$ or $\de_{2}^{0}\circ \bi_0^1$ and the latter is equal to the former by the bracelet relation $(2)$.

        If $\br(C)=0$, then $C$ is a either the identity or the composition of twists. If $C'$ is also of Type II with $t_0(C)=t_0(C')$, then $C$ and $C'$ are related by relation (8) in~\cref{thm: minimal list of generators}. Note that one can take $\tw_k^{t_0(C)}$ as the representative of $[C]$.
\end{proof}

\begin{theorem}\label{thm: unique decomp}
    Every morphism $[C]\in \cyl$ has a decomposition as
    \[ C= C_{III}\circ C_{II}\circ C_{I},\]
    which is unique up to the relations from~\cref{thm: minimal list of generators}.
\end{theorem}

\begin{proof}
    By~\cref{thm: gens of cyl}, any cobordism $C:S^1_m\to S^1_n$ in $\cyl$ can be written as a composition of the generators $\tw_k, \bi_k^i, \de_k^i, \id_k$. 
    
    Whenever a $\de_k^i$ occurs in this composition, move it as far to the beginning of the cobordism (to the left in the picture; to the right in the factorization) as possible. Assuming $\de_k^i$ is not already adjacent to a death, the following may happen: the death will cancel with a birth (relations (1), (3), (3*)); or move past a birth or twist (relations (4), (4*)); or will be unable to move past a birth (as in the bracelet relation (2) or the contractible circle relation (1)). This gives a representation of $[C]$ where all the deaths occur to the left of any other generators, except the deaths involved in bracelets.

    An analogous argument moves all the births occurring in the decomposition of $C$ to the end of the cobordism (to the right in the picture; to the left in the factorization), besides the births involved in bracelets.
    This leaves compositions of twists and compositions of bracelets in the center of the cobordism.  

    Let $C_I$ be the unique cylinder cobordism of Type I, where $\dind{C_I}=\dind{C}$; this is unique up to the relations of \cref{thm: minimal list of generators} by~\cref{lem: types uniquely det by invariants}; similarly for $C_{III}$. If $\br(C)\neq 0$, then $C_{II}$ is uniquely determined: since we have moved all births and deaths past each other, we must have $C_{II}=(\de_{2}^{1}\circ \bi_0^0)^{\br(C)}$. If $\TT(C)\neq 0$, then to determine $C_{II}$ as in~\cref{lem: type II from invariants}, we just need to specify $t_0(C_{II})$. This number is uniquely determined by the equation $t_0(C) \equiv_{\tau(C)} t_0(C_I) + t_0(C_{II}) + t_0(C_{III})$.
    We then have $C=C_{III}\circ C_{II}\circ C_{I}$.
\end{proof}

\begin{corollary}\label{cor:cyl gens and relns}
        The list of relations from~\cref{thm: minimal list of generators}  
        are a complete generating set of relations for morphisms in $\cyl$. 
\end{corollary}

\section{$\cyl$-objects and Temperley-Lieb algebras}\label{section6:TQFTs}

We can now leverage our understanding of the generators and relations in $\cyl$ to identify properties of its representations. A \textit{representation} of $\cyl$, or \textit{$\cyl$-object}, is a functor $\cyl\to \cat C$ where $\cat C$ is any category. As a corollary of \cref{cor:cyl gens and relns}, we have the following classification of $\cyl$-objects.

\begin{corollary}\label{cor:cyl rep data}
    A functor $\cyl\to \cat C$ is specified by the following data:\begin{itemize}
        \item for each $n\geq 0$, an object $c_n\in \cat C$,
        \item for each $n\geq 0$, an isomorphism $t_n\colon c_n\to c_n$,
        \item for each $n\geq 2$, maps $d_n^i\colon c_n\to c_{n-2}$ for $0\leq i\leq n-1$,
        \item for each $n\geq 0$, maps $s_n^j\colon c_n \to c_{n+2}$ for $0\leq j \leq n+1$,
    \end{itemize}
    subject to the following relations:
    \begin{enumerate}
        \item[(i)] $d_{k-2}^i\circ d_{k}^j = d_{k-2}^{j-2}\circ d_{k}^{i} $ for $ i < j - 1 $,
        \item[(ii)] $s_{k+2}^i\circ s_{k}^j=s_{k+2}^{j+2}\circ 
        s_k^{i}$ if $i \leq j$,
        \item[(iii)] $d_{n+2}^j\circ s_n^i = \left\{\begin{array}{cc}
             \id & i=j-1,j,j+1; \\
             s_{n-2}^{j-2} \circ d^i_n & i<j-1;\\
             s_{n-2}^j \circ d_n^{i-2} & i> j+1,
        \end{array}\right.$
        \item[(iv)] $t_n^n = \id$,
        \item[(v)] $t_{n+2}\circ s_n^j = s_n^{j+1}\circ t_n$,
        \item[(vi)] $t_n\circ d_{n+2}^i = d_{n+2}^{i+1}\circ t_{n+2}$,
    \end{enumerate}
\end{corollary}

\begin{remark}
    In light of~\cref{rmk:Cob 2cat}, one could consider a version of $\cyl$-objects that arise from $2$-functors. In this extended setting, if $f\colon X\to X$ is the $1$-morphism in $\cat C$ assigned to the closed interval with one marked point, then $c_n$ is the trace of the $n$-fold composition of $f$. It would be interesting to explicate exactly when $f$ describes a $\cyl$-object (e.g. when $\cat C$ is the Morita $2$-category of algebras, bimodules, and intertwiners) and further understand the connection between $2$-categorical $\cyl$-objects and categorical traces (see \cite{PontoShulman}). We thank an anonymous referee for this suggestion.
\end{remark}

Note that the relations (i)--(iii) are very similar to the data of a simplicial object in $\cat C$, while relations (iv)--(vi) are similar to the additional structure of a cyclic object. In this section, we will further unpack this structure and discuss how $\cyl$-representations relate to other ideas in the literature, such as the affine Temperley-Lieb algebras of Graham--Lehrer \cite{Graham-Lehrer:98} as well as the annular Temperley-Lieb algebras of Jones \cite{jones:01, Jones} and its connection to Connes' cyclic category \cite{penneys}.

Our work of finding the generators and relations for $\cyl$ can be seen as an extension of \cite{erdmann/green:98}, where Erdmann--Green give generators and relations for affine Temperley-Lieb algebras studied in \cite{fan/green:97, green:98}, which are related to~\cite{Graham-Lehrer:98}. 
Our work is also similar to some of the results of \cite{penneys}, wherein Penneys obtains explicit generators and relations for \textit{annular Temperley-Lieb algebras}, as described by Jones \cite{jones:01, Jones}. 
The defining diagrams for annular Temperley-Lieb algebras come with a ``shading,'' as described in \cref{sec:annular TLA}, whereas affine Temperley-Lieb algebras do not require the analogous diagrams to be shaded.
Our work can be viewed as an ``unshaded'' analog of Penneys' results.

For the upcoming discussion, it will be helpful to note some properties of the category $\cyl$. Recall from \cref{defn:cyla and cyl} that $\cyl^a$ is the cylindrical nested cobordism category where we also keep track of the number of contractible closed loops (but not how these loops are embedded in the cylinder).

\begin{remark}\label{rmk:gens and relns for cyla}
    Essentially the same proofs from \cref{section3:cyl} work to describe generators and relations for $\cyl^a$. In particular, $\Hom(\cyl^a)$ has the same generators as $\cyl$ (with no contractible circles) and additional generators $(\id_k, 1)$, $k\geq 0$, which has one contractible circle. The generators are subject to the relations analogous to (2)--(8) in \cref{thm: minimal list of generators}, and relation (1) is replaced with $(\de^i_{k+2}, 0)\circ (\bi^i_k, 0) = (\id_k, 1)$. Consequently, a functor $\cyl^a\to \cat C$ is specified by the same data as a $\cyl$-object, along with an endomorphism $\chi_k\colon c_k\to c_k$ for $k\geq 0$ which is the image of $(\id_k, 1)$. The only relation that changes is that $d^j_{n+2}\circ s^j_n = \chi_k$; the map $\chi_0$ can be viewed as an Euler characteristic (trace of the identity map), as we will discuss in \cref{sec:bar construction}.
\end{remark}

The following observation will also be helpful for some of our comparisons.

\begin{definition}\label{defn:cyl0 and cyl1}
     We can write $\cyl = \cyl_0\amalg \cyl_1$, where $\cyl_0$ is the full subcategory on even-parity objects, $\{S^1_{2k}\}_{k\geq 0}$, and $\cyl_1$ is the full subcategory on odd-parity objects, $\{S^1_{2k+1}\}_{k\geq 0}$. Similarly, $\cyl^a = \cyl^a_0\amalg \cyl^a_1$.
\end{definition}

We show that the cyclic category $\Lambda$ includes into $\cyl_0$, so every $\cyl_0$-object has an underlying cyclic object. In fact, a $\cyl_0$-object can be seen as a cyclic object where the cyclic action ``has square roots,'' and we make this idea precise by introducing a category $\sqrt{\Lambda}$ (\cref{defn:sqrt Lambda}) and an inclusion $\Lambda\to \sqrt{\Lambda}$. The category $\sqrt{\Lambda}$ is similar in spirit to the $C_2$-twisted cyclic category used to define $C_2$-twisted topological Hochschild homology \cite{BHM:93, ABGHLM:18}. Just as edgewise subdivision turns cyclic objects into $C_2$-twisted ones, we introduce a \textit{doubling construction} that turns a cyclic object into a $\sqrt{\Lambda}$-object. Inspired by the cyclic bar construction, we define the \textit{$\cyl$-bar construction} (\cref{defn:bar construction}), which takes as input a self-dual object in a strict monoidal category $\cat C$ and produces a $\cyl$-object in $\cat C$.

\subsection{Connection to affine Temperley-Lieb algebras}\label{section5:TLalgebras}

In \cite{Graham-Lehrer:98}, Graham and Lehrer introduce the \textit{affine Temperley-Lieb category}, denoted $\cat T^a$. 
A functor $W\colon \cat T^a\to {\rm Mod}_R$ (the category of $R$-modules for some ring $R$) gives rise to representations of affine Temperley-Lieb algebras, called \textit{cell modules} (or \textit{Weyl modules}). In this section, we briefly review these definitions and discuss how the category $\cat T^a$ is related to $\cyl$.

To define the affine Temperley-Lieb category, we first define a category of diagrams $\cat D^a$. The objects of $\cat D^a$ are non-negative integers and morphisms involve \textit{affine diagrams}. An affine diagram $n\to m$ can be visualized as two infinite horizontal rows of nodes on the grid $\mathbb{Z}\times\{0,1\}\subseteq \RR\times \RR$, along with edges between them satisfying the following:
\begin{itemize}
    \item every node is the endpoint of exactly one edge,
    \item no edges intersect,
    \item every edge lies within  $\RR\times[0,1]$,
    \item the diagram is invariant under the shift $(t,k)\mapsto (t+n, k+m)$,
    \item every edge either connects two points or does not meet any node, in which case it is an infinite horizontal line. There are only finitely many (possibly zero) edges of this second type.
\end{itemize}
Composition is defined by stacking of diagrams. This category comes with a natural involution, giving by flipping the diagram. These definitions (and the ones that follow) can be made more rigorous using constructions on finite totally ordered sets, see \cite[\S 1]{Graham-Lehrer:98} for details. The following figure shows an example of a composition of affine diagrams, $3\xrightarrow{\alpha} 3\xrightarrow{\beta} 1$.
\[
\begin{tikzpicture}[scale=0.75]
    \node at (-4,0.5) {$\dots$};
    \fill (-3,0.5) circle (3pt);
    \fill (-2,0.5) circle (3pt);
    \fill (-1,0.5) circle (3pt);
    \fill (0,0.5) circle (3pt);
    \fill (1,0.5) circle (3pt);
    \fill (2,0.5) circle (3pt);
    \fill (3,0.5) circle (3pt);
    \node at (4,0.5) {$\dots$};
    \node at (-4,1.5) {$\dots$};
    \fill (-3,1.5) circle (3pt);
    \fill (-2,1.5) circle (3pt);
    \fill (-1,1.5) circle (3pt);
    \fill (0,1.5) circle (3pt);
    \fill (1,1.5) circle (3pt);
    \fill (2,1.5) circle (3pt);
    \fill (3,1.5) circle (3pt);
    \node at (4,1.5) {$\dots$};
    \draw[dashed] (1.5,1.75) -- (1.5,0.25);
    \draw[dashed] (-1.5,1.75) -- (-1.5,0.25);
    \draw[thick] (1,1.5) -- (0,0.5);
    \draw[thick] (-2,1.5) -- (-3,0.5);
    \draw[thick] (3.75,1.5) -- (3,0.5);
    \draw[thick] (-1,0.5) arc (0:180:0.5);
    \draw[thick] (2,0.5) arc (0:180:0.5);
    \draw[thick] (-1,1.5) arc (180:360:0.5);
    \draw[thick] (2,1.5) arc (180:360:0.5);

    \node at (-4,-0.5) {$\dots$};
    \fill (-3,-0.5) circle (3pt);
    \fill (-2,-0.5) circle (3pt);
    \fill (-1,-0.5) circle (3pt);
    \fill (0,-0.5) circle (3pt);
    \fill (1,-0.5) circle (3pt);
    \fill (2,-0.5) circle (3pt);
    \fill (3,-0.5) circle (3pt);
    \node at (4,-0.5) {$\dots$};
    \node at (-4,-1.5) {$\dots$};
    \fill (-3,-1.5) circle (3pt);
    \fill (0,-1.5) circle (3pt);
    \fill (3,-1.5) circle (3pt);
    \node at (4,-1.5) {$\dots$};
    \draw[dashed] (1.5,-0.25) -- (1.5,-1.75);
    \draw[dashed] (-1.5,-0.25) -- (-1.5,-1.75);
    \draw[thick] (-3,-0.5) -- (-3,-1.5);
    \draw[thick] (0,-0.5) -- (0,-1.5);
    \draw[thick] (3,-0.5) -- (3,-1.5);
    \draw[thick] (-2,-0.5) arc (180:360:0.5);
    \draw[thick] (1,-0.5) arc (180:360:0.5);

\node at (-5,0.5) {$\alpha$};
\node at (-5,-0.5) {$\beta$};
\draw[dashed] (-5,0) -- (4,0);
\node at (5,0) {$=$};
\node at (10,-1.5) {$\beta\circ \alpha$};
    \node at (6,-0.5) {$\dots$};
    \fill (7,-0.5) circle (3pt);
    \fill (10,-0.5) circle (3pt);
    \fill (13,-0.5) circle (3pt);
    \node at (14,-0.5) {$\dots$};
    \node at (6,0.5) {$\dots$};
    \fill (7,0.5) circle (3pt);
    \fill (8,0.5) circle (3pt);
    \fill (9,0.5) circle (3pt);
    \fill (10,0.5) circle (3pt);
    \fill (11,0.5) circle (3pt);
    \fill (12,0.5) circle (3pt);
    \fill (13,0.5) circle (3pt);
    \node at (14,0.5) {$\dots$};
    \draw[dashed] (8.5,0.75) -- (8.5,-0.75);
    \draw[dashed] (11.5,0.75) -- (11.5,-0.75);
    \draw[thick] (7,0.5) -- (7,-0.5);
    \draw[thick] (10,0.5) -- (10,-0.5);
    \draw[thick] (13,0.5) -- (13,-0.5);
    \draw[thick] (8,0.5) arc (180:360:0.5);
    \draw[thick] (11,0.5) arc (180:360:0.5);
    \draw[thick] (8.75,-0.5) arc (0:360:0.25);
    \draw[thick] (11.75,-0.5) arc (0:360:0.25);
\end{tikzpicture} 
\]
Affine diagrams do not have contractible loops, but composition can introduce such loops, as shown in the example above. Let $\mu(\alpha, \beta)$ denote the number of contractible loops in the diagram composition $\alpha\circ \beta$ of two affine diagrams $\alpha$ and $\beta$. A morphism in $\cat D^a$ is 
more precisely an affine diagram (with no contractible loops) along with a non-negative integer $\mu$, standing in for the number of non-contractible loops.

As discussed in \cite{Graham-Lehrer:98}, an affine diagram can also be visualized as a striped cylinder cobordism. However, the notion of equivalence of striped cylinder cobordism in \cref{defn:equiv reln on nested cobs} is not the same as the notion of equivalence used for affine diagrams; two affine diagrams
are \textit{affine equivalent} if there is a diffeomorphism of the ambient $2$-cylinder that restricts to an isotopy on the embedded $1$-manifold. In particular, Dehn twists are not affine equivalent to the identity. The discussion following \cite[Definition 1.3]{Graham-Lehrer:98}, shows that $\cat D^a$ is isomorphic to this variation on $\cyl^a$. In the remainder of this subsection, we freely make use of this identification.

\begin{corollary}\label{cor:cyl and Da}
    There is a functor $\cat D^a\to \cyl^a$ given on objects by $n\mapsto S^1_n$ and on morphisms by sending $(\alpha, \mu)\to [\alpha]$, where $\alpha\sim \beta$ is generated by the ``Dehn twist'' relation that identifies the affine diagram $n\to n$ with edges between the nodes $(k,0)$ and $(k+n,1)$, $k\in \mathbb{Z}$, with the identity.
\end{corollary}

\begin{remark}\label{rmk:Da gens relns}
In light of \cref{rmk:gens and relns for cyla}, we can obtain a generators and relations description for $\cat D^a$. In particular, we remove relation (8) from \cref{thm: minimal list of generators} and add in a generator $\tw^{-1}_k$ for $k\geq 2$. The quotient $\cat D^a\to \cyl^a$ is given by imposing relation (8).
\end{remark}

Now fix a ring $R$ and an element $q\in R^{\times}$, and set $\delta:=-(q + q^{-1})$. To go from the diagram category $\cat D^a$ to the Temperley-Lieb category $\cat T^a$, we freely enrich over the category of $R$-modules. The following definition is \cite[Definition 2.5]{Graham-Lehrer:98}.

\begin{definition}
    The \textit{affine Temperley-Lieb category} $\cat T^a$ is the category $\cat D^a$ but freely enriched over $R$-modules, i.e. $\cat T^a$ has  
    objects non-negative integers and $\cat T^a(n,m)$ is the free $R$-module on $\cat D^a(n,m)$. For $\alpha$ and $\beta $ composable morphisms in $D^a$, composition is given by $\alpha\beta = \delta^{\mu(\alpha, \beta)}\alpha \circ \beta$, where $\alpha \circ \beta$ is the composition of affine diagrams in $D^a$. Composition is extended $R$-bilinearly to all of $\cat T^a$.
\end{definition}

\begin{definition}
    The \textit{affine Temperley-Lieb algebra} $\cat T^a(n)$ is the $R$-module $\cat T^a(n,n)$.
    A \textit{representation} of $\cat T^a$ (or $\cat T^a$-module) is a functor $F\colon \cat T^a\to {\rm Mod}(R)$.
\end{definition}

A representation $F$ determines representations of all the (affine) Temperley-Lieb algebras simultaneously, for all $n\geq 0$. 

\begin{remark}
    The category ${\rm Mod}(R)$ is naturally enriched over itself, and an enriched functor $F\colon \cat T^a\to {\rm Mod}(R)$ is equivalent to an ordinary functor $\cat D^a\to {\rm Mod}(R)$. Many examples, such as  \cite[Definition 2.6]{Graham-Lehrer:98}, are constructed this way.
\end{remark}

The composition $\cat D^a\to \cyl^a\to \cyl$ implies that representations of $\cyl$ provide one source of $\cat D^a$-modules (and hence $\cat T^a$-modules).

\begin{corollary}\label{cor:cyl object is affine TLA}
    Every $\cyl$-object is a $\cat T^a$-module.
\end{corollary}

As a corollary of \cref{rmk:Da gens relns}, one could also obtain a concrete list of data needed to build a $\cat D^a$-module, similar to that of \cref{cor:cyl rep data}.
\subsection{Connection to annular Temperley-Lieb algebras}\label{sec:annular TLA}

In \cite{penneys}, Penneys defines an abstract version $a\Delta$ of Jones' annular Temperley-Lieb category ${\rm \bf Atl}$ and shows that $a\Delta\cong {\rm \bf Atl}$ as involutive categories. In this section, we briefly recall these categories and describe how $\bf{Atl}$ is related to $\cyl$. We adopt the notation of \cite[Section 2]{penneys} throughout.

The objects of $\bf{Atl}$ are $[n]$ for $n\in \mathbb{Z}_{\geq 1}$ along with two additional objects $[0^+]$ and $[0^-]$. For $n>0$, we can visualize $[n]$ as a circle with $2n$ marked points. 

The morphisms of the category $\bf{Atl}$ are constructed from \textit{$(m,n)$-tangles}, roughly defined as follows. An $(m,n)$-tangle $T$ is an annulus in the complex plane whose \textit{outer boundary} is the unit circle $D_0(T)$ and whose \textit{inner boundary} $D_1(T)$ is the circle of radius $1/4$. The inner boundary has $2m$ marked points and the outer boundary has $2n$ marked points, and every marked point meets exactly one string (a smoothly embedded curve in the annulus, transverse to the boundary circles). 
A string is either a closed curve (a \textit{loop}) 
or connects two marked points, and the strings do not intersect one another. Each region of the annulus is either shaded or unshaded, so that regions which share a string as a boundary have different shadings. Finally, both $D_0(T)$ and $D_1(T)$ come with a marked unshaded region, picked out by distinguishing a ``simple interval'' between two adjacent boundary points.

\begin{figure}[h!]
    \centering
\includegraphics[scale=0.11]{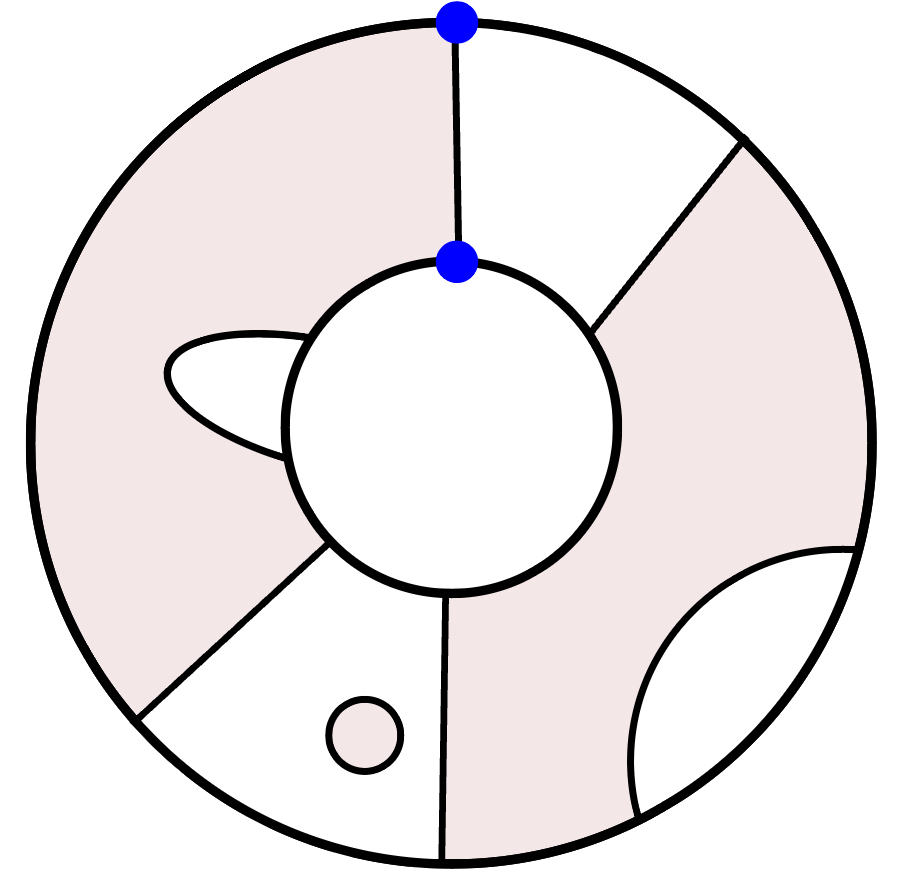}
    \caption{Example of a tangle.}
    \label{fig:tangle}
\end{figure}

Two annular tangles are said to be equivalent if there is a orientation-preserving diffeomorphism between the two. Composition of tangles is given by nesting annuli, after isotoping the strings to line up the marked regions.

\begin{figure}[h!]
    \centering
    \includegraphics[scale=0.15]{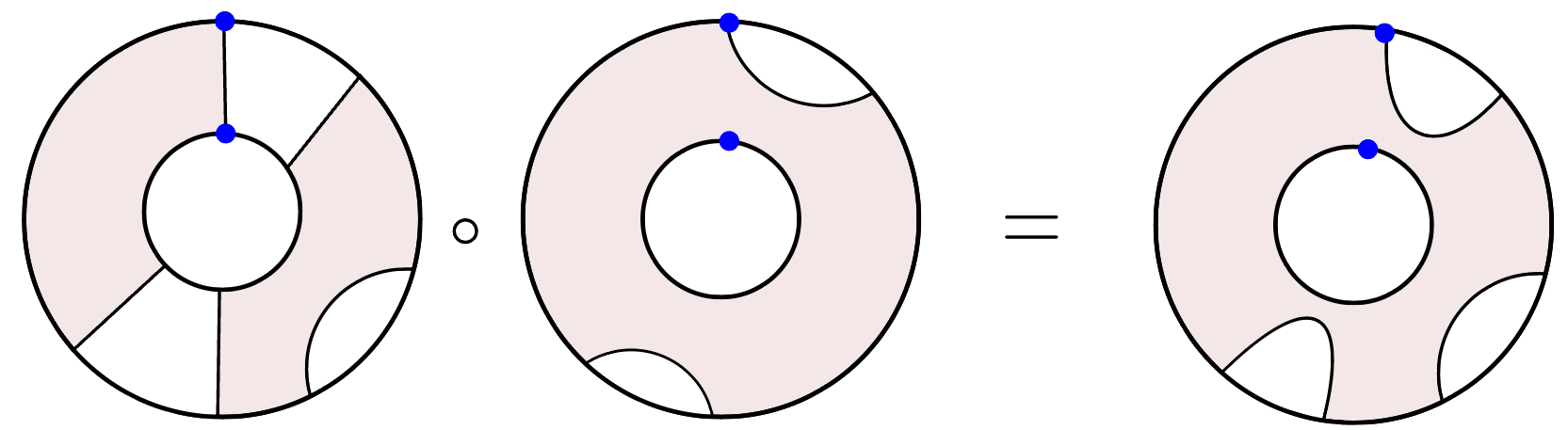}
    \caption{The composition of tangles in $\bf{Atl}$.}
    \label{fig:composition_tangle}
\end{figure}

\begin{definition}
The objects of $\bf{Atl}$ are $\mathbb{Z}_{\geq 1}\amalg\{0^{\pm}\}$. A morphism $[n]\to [m]$ in $\bf{Atl}$ is a triple $(T, c_+, c_-)$ where $T$ is an equivalence class of $(n,m)$-tangles and $c_+, c_-\in \mathbb{Z}_{\geq 0}$ denote the number of closed unshaded and shaded loops, respectively. Composition is given as described above.
\end{definition}

To relate $\bf{Atl}$ to $\cyl$, recall from \cref{defn:cyl0 and cyl1} that $\cyl = \cyl_0\amalg \cyl_1$ and $\cyl^a=\cyl^a_0\amalg \cyl^a_1$ can be partitioned into odd and even parts.

\begin{proposition}
    There is a functor ${\bf{Atl}}\to \cyl^a_0 \to  \cyl_0$ given by forgetting the shading and subsequently the contractible closed loops.
\end{proposition}
\begin{proof}
    The second functor is the one appearing in \cref{defn:cyla and cyl}. Using the generators and relations from \cite[\S 2.2]{penneys}, we send $[n]\in {\rm \bf Atl}$ to the object $S^1_{2n}$, sending the left point in the marked simple interval to the marked point in $S^1_{2n}$. The generators are assigned as follows:\begin{itemize}
        \item $a_i\in {\bf{Atl}}(n,n-1)$ maps to $\de^{2i}_{2n}\in \cyl^a(2n, 2n-2)$,
        \item $b_i\in {\bf{Atl}}(n,n+1)$ maps to $\bi^{2i}_{2n}\in \cyl^a(2n, 2n+2)$,
        \item $t\in {\bf Atl}(n,n)$ maps to $\tw^2_{2n}\in \cyl^a(2n,2n)$,
        \item $(\id_{[n]}, j, k)\in {\bf Atl}(n,n)$ is sent to $(\id_{2n}, j+k) = (\id_{2n}, 1)^{j+k}\in \cyl^a(2n,2n)$ for all $n\geq 0$ and $j,k\in \mathbb{N}$. In particular $\id_{[0+]}$ and $\id_{[0-]}$ are both sent to $\id_0$. 
    \end{itemize}
Checking our relations in \cref{rmk:gens and relns for cyla} against those in \cite[Theorem 2.20]{penneys} shows that this assignment is indeed functorial.
\end{proof}

The only morphism that is not sent to ``itself'' is the twist $t$. Since Penneys' tangles are shaded, his twist operation has to preserve the shading, so $t\colon [n]\to [n]$ is sent to ${\rm tw}_{2n}^2\colon S^1_{2n}\Rightarrow S^1_{2n}$. Hence the primary difference between our definition and Penneys' is that we have more twists.

\begin{definition}
    A \textit{shaded annular object} in $\cat C$ is a functor $\bf{Atl}\to \cat C$.
\end{definition}

These are annular Temperley-Lieb algebras, as in \cite{jones:01}, when $\cat C$ is a category of $R$-modules. These are also closely related to affine Temperley-Lieb algebras \cite{Graham-Lehrer:98} from the previous subsection. 
Restricting along the composition ${\bf Atl}\to \cyl^a_0\to \cyl_0$ implies the following result.

\begin{corollary}\label{cor:cyl0 obj is annular TLA}
    Every functor $\cyl_0\to \cat C$ is a shaded annular $\cat C$-object.
\end{corollary}

\subsection{Connection to cyclic objects}\label{sec:cyclic objects}
A main result of \cite{penneys} is a new proof of Jones's result that $\bf{Atl}$ can be obtained from two copies of the cyclic category $\Lambda$, glued together over the groupoid $\amalg_{n \neq 0} C_n$ of cyclic groups, with some minor adjustments. In our case, the connection between $\cyl_0$ and $\Lambda$ is a bit simpler.

There are a few equivalent ways to define $\Lambda$ \cite{connes, loday:92}, but for our purposes, the most helpful description is the following.

\begin{definition}\label{defn:cyclic cat}
    The category $\Lambda^{\rm op}$ has objects $[n]$ for $n\in \mathbb{Z}_{\geq 0}$. Morphisms are generated by\begin{itemize}
        \item[] $d_n^i\colon [n]\to [n-1]$ for $n\geq 1$ and $0\leq i\leq n$,
        \item[] $s_n^j\colon [n]\to [n+1]$ for $n\geq 0$ and $0\leq j\leq n$,
        \item[] $t_n\colon [n]\to [n]$ for $n\geq 0$
    \end{itemize} where $d^i_n$ and $s^j_n$ are the face and degeneracy maps from $\Delta^{\rm op}$ and $t_n(x) = x+1$ for $x\neq n$ and $t_n(n)=0$. These generators are subject to the simplicial relations \begin{itemize}
        \item[(i)] $d^i_n\circ d^j_{n+1} = d^{j-1}_{n}\circ d_{n+1}^{i}$ for $i<j$,
        \item[(ii)] $s_n^i\circ s_{n-1}^j = s_n^{j+1} \circ s_{n-1}^{i}$ for $i\leq j$,
        \item[(iii)] $d_{n+1}^i\circ s_n^j = \left\{\begin{array}{cc}
            s_{n-1}^{j-1} \circ d^i_n & i<j;\\
             \id & i=j,j+1; \\
             s_{n-1}^j \circ d_n^{i-1} & i> j+1,
        \end{array}\right.$
        \end{itemize}
        and the cyclic relations
        \begin{itemize}
        \item[(iv)] $t_n^{n+1} = \id$,
        \item[(v)] $t_{n+1}\circ s_n^j = s_n^{j+1}\circ t_n$,
        \item[(vi)] $t_n\circ d_{n+1}^i = d_{n+1}^{i+1}\circ t_{n+1}$.
    \end{itemize}
\end{definition}

The cyclic category can also be described visually \cite{malkiewich:15}, where $[n]=\{0,1,\dots, n\}$ is thought of as a circle with $n+1$ marked points with cyclic labeling and morphisms are annular diagrams, as below. These visualizations are helpful for the comparison between $\Lambda$ and $\cyl_0$.
\[
\begin{tikzpicture}
\coordinate (center) at (0,0);
\def\radius{1cm}
  \draw[thick] (center) circle[radius=\radius];

    \fill[blue] (center)+(90:\radius) circle[radius=3pt];
    \node at (90:1.5cm) {$0$};
    \fill (center)+(35:\radius) circle[radius=3pt];
    \node at (35:1.5cm) {$1$};
    \fill (center)+(-20:\radius) circle[radius=3pt];
    \node at (-20:1.5cm) {$2$};
    \node at (-60:1.5cm) {$\cdot$};
    \node at (-70:1.5cm) {$\cdot$};
    \node at (-80:1.5cm) {$\cdot$};
    \fill (center)+(145:\radius) circle[radius=3pt];
    \node at (145:1.5cm) {$n$};
        \fill (center)+(145:\radius) circle[radius=3pt];
        \fill (center)+(190:\radius) circle[radius=3pt];
    \node at (190:1.75cm) {$n-1$};

    \node at (270:2cm) {$[n]$};
\end{tikzpicture}
\hspace{0.5cm}
\begin{tikzpicture}
\coordinate (center) at (0,0);
\def\radius{1cm}
  \draw[thick] (center) circle[radius=1cm];

    \fill[blue] (center)+(90:\radius) circle[radius=2pt];
    \fill (center)+(18:\radius) circle[radius=2pt];
    \fill (center)+(162:\radius) circle[radius=2pt];
    \fill (center)+(234:\radius) circle[radius=2pt];
    \fill (center)+(306:\radius) circle[radius=2pt];

  \draw[thick] (center) circle[radius=0.25cm];
    \fill[blue] (center)+(90:0.25cm) circle[radius=2pt];
    \fill (center)+(18:0.25cm) circle[radius=2pt];
    \fill (center)+(162:0.25cm) circle[radius=2pt];
    \fill (center)+(234:0.25cm) circle[radius=2pt];
    \fill (center)+(306:0.25cm) circle[radius=2pt];
    
    \draw (90:0.25cm) --  (18:\radius);
    \draw (18:0.25cm) --  (306:\radius);
    \draw (306:0.25cm) --  (234:\radius);
    \draw (234:0.25cm) --  (162:\radius);
    \draw (162:0.25cm) --  (90:\radius);

    \node at (270:2cm) {$t_4\colon [4]\xrightarrow{\cong} [4]$};
\end{tikzpicture}
\hspace{0.75cm}
\begin{tikzpicture}
\coordinate (center) at (0,0);
\def\radius{1cm}
  \draw[thick] (center) circle[radius=1cm];

    \fill[blue] (center)+(90:\radius) circle[radius=2pt];
    \fill (center)+(18:\radius) circle[radius=2pt];
    \fill (center)+(162:\radius) circle[radius=2pt];
    \fill (center)+(234:\radius) circle[radius=2pt];
    \fill (center)+(306:\radius) circle[radius=2pt];

  \draw[thick] (center) circle[radius=0.25cm];
    \fill[blue] (center)+(90:0.25cm) circle[radius=2pt];
    \fill (center)+(180:0.25cm) circle[radius=2pt];
    \fill (center)+(270:0.25cm) circle[radius=2pt];
    \fill (center)+(0:0.25cm) circle[radius=2pt];
    
    \draw[blue] (90:0.25cm) --  (90:\radius);
    \draw (180:0.25cm) --  (162:\radius);
    \draw (270:0.25cm) --  (234:\radius);
    \draw (0:0.25cm) --  (18:\radius);

    \node at (270:2cm) {$s_3^2\colon [3]\to [4]$};
\end{tikzpicture}
\hspace{0.75cm}
\begin{tikzpicture}
\coordinate (center) at (0,0);
\def\radius{1cm}
  \draw[thick] (center) circle[radius=1cm];

    \fill[blue] (center)+(90:\radius) circle[radius=2pt];
    \fill (center)+(180:\radius) circle[radius=2pt];
    \fill (center)+(270:\radius) circle[radius=2pt];
    \fill (center)+(0:\radius) circle[radius=2pt];

  \draw[thick] (center) circle[radius=0.25cm];
    \fill[blue] (center)+(90:0.25cm) circle[radius=2pt];
    \fill (center)+(18:0.25cm) circle[radius=2pt];
    \fill (center)+(162:0.25cm) circle[radius=2pt];
    \fill (center)+(234:0.25cm) circle[radius=2pt];
    \fill (center)+(306:0.25cm) circle[radius=2pt];
    
    \draw (90:0.25cm) --  (90:\radius);
    \draw (162:0.25cm) --  (180:\radius);
    \draw (306:0.25cm) --  (270:\radius);
    \draw (18:0.25cm) --  (0:\radius);
    \draw (234:0.25cm) --  (270:\radius);
    
    \node at (270:2cm) {$d_4^2\colon [4]\to [3]$};
\end{tikzpicture}
\]
Composition is again given by nesting the annular diagrams, as shown for $d_4^2\circ t_4$ in the figure below.
\[
\begin{tikzpicture}[scale=0.75]
\coordinate (center) at (0,0);
\def\radius{1cm}
  \draw[thick] (center) circle[radius=1cm];

    \fill[blue] (center)+(90:\radius) circle[radius=2pt];
    \fill (center)+(18:\radius) circle[radius=2pt];
    \fill (center)+(162:\radius) circle[radius=2pt];
    \fill (center)+(234:\radius) circle[radius=2pt];
    \fill (center)+(306:\radius) circle[radius=2pt];

  \draw[thick] (center) circle[radius=0.25cm];
    \fill[blue] (center)+(90:0.25cm) circle[radius=2pt];
    \fill (center)+(18:0.25cm) circle[radius=2pt];
    \fill (center)+(162:0.25cm) circle[radius=2pt];
    \fill (center)+(234:0.25cm) circle[radius=2pt];
    \fill (center)+(306:0.25cm) circle[radius=2pt];
    
    \draw (90:0.25cm) --  (18:\radius);
    \draw (18:0.25cm) --  (306:\radius);
    \draw (306:0.25cm) --  (234:\radius);
    \draw (234:0.25cm) --  (162:\radius);
    \draw (162:0.25cm) --  (90:\radius);

  \draw[thick] (center) circle[radius=1.75cm];
    \fill[blue] (center)+(90:1.75cm) circle[radius=2pt];
    \fill (center)+(180:1.75cm) circle[radius=2pt];
    \fill (center)+(270:1.75cm) circle[radius=2pt];
    \fill (center)+(0:1.75cm) circle[radius=2pt];

    \draw (90:1.75cm) --  (90:\radius);
    \draw (0:1.75cm) --  (18:\radius);
    \draw (306:0.25cm) --  (234:\radius);
    \draw (180:1.75cm) --  (162:\radius);
    \draw (270:1.75cm) --  (234:\radius);
    \draw (270:1.75cm) --  (306:\radius);
\end{tikzpicture} \begin{tikzpicture}[scale=0.75]
\coordinate (center) at (0,0);
\def\radius{1cm}

  \draw[thick] (center) circle[radius=0.25cm];
    \fill[blue] (center)+(90:0.25cm) circle[radius=2pt];
    \fill (center)+(18:0.25cm) circle[radius=2pt];
    \fill (center)+(162:0.25cm) circle[radius=2pt];
    \fill (center)+(234:0.25cm) circle[radius=2pt];
    \fill (center)+(306:0.25cm) circle[radius=2pt];

  \draw[thick] (center) circle[radius=1.75cm];
    \fill[blue] (center)+(90:1.75cm) circle[radius=2pt];
    \fill (center)+(180:1.75cm) circle[radius=2pt];
    \fill (center)+(270:1.75cm) circle[radius=2pt];
    \fill (center)+(0:1.75cm) circle[radius=2pt];

\draw (90:1.75cm) --  (162:0.25cm);
\draw (180:1.75cm) --  (234:0.25cm);
\draw (270:1.75cm) --  (306:0.25cm);
\draw (0:1.75cm) --  (90:0.25cm);

\draw (270:1.75cm) to[out=20,in=-20]  (18:0.25cm);

\node at (180:2.15cm) {$=$};
\end{tikzpicture}
\]

\begin{theorem}\label{thm:cyl and cyclic cat}
    There is an inclusion of categories $\Lambda^{\rm op}\hookrightarrow\cyl$.
\end{theorem}\begin{proof}
    The inclusion $\Lambda\to \cyl_0$ is given by sending $[n]\mapsto S^1_{2(n+1)}$, $d_n^i\mapsto d_{2n+2}^{2i}$, $s_n^j\mapsto b_{2n+2}^{2j}$, and $t_n\mapsto t_{2n+2}^2$; see Figure \ref{fig:duals} for an example. 
    The claim follows by checking the generators and relations for $\cyl_0$ from \cref{thm: minimal list of generators} against those in \cref{defn:cyclic cat}.
\end{proof}

\begin{figure}[ht]
\adjustbox{scale=0.55}{

\tikzset{every picture/.style={line width=0.75pt}} 

\begin{tikzpicture}[x=0.75pt,y=0.75pt,yscale=-1,xscale=1]

\draw (323.5,260.5) node  {\includegraphics[width=405.75pt,height=333.75pt]{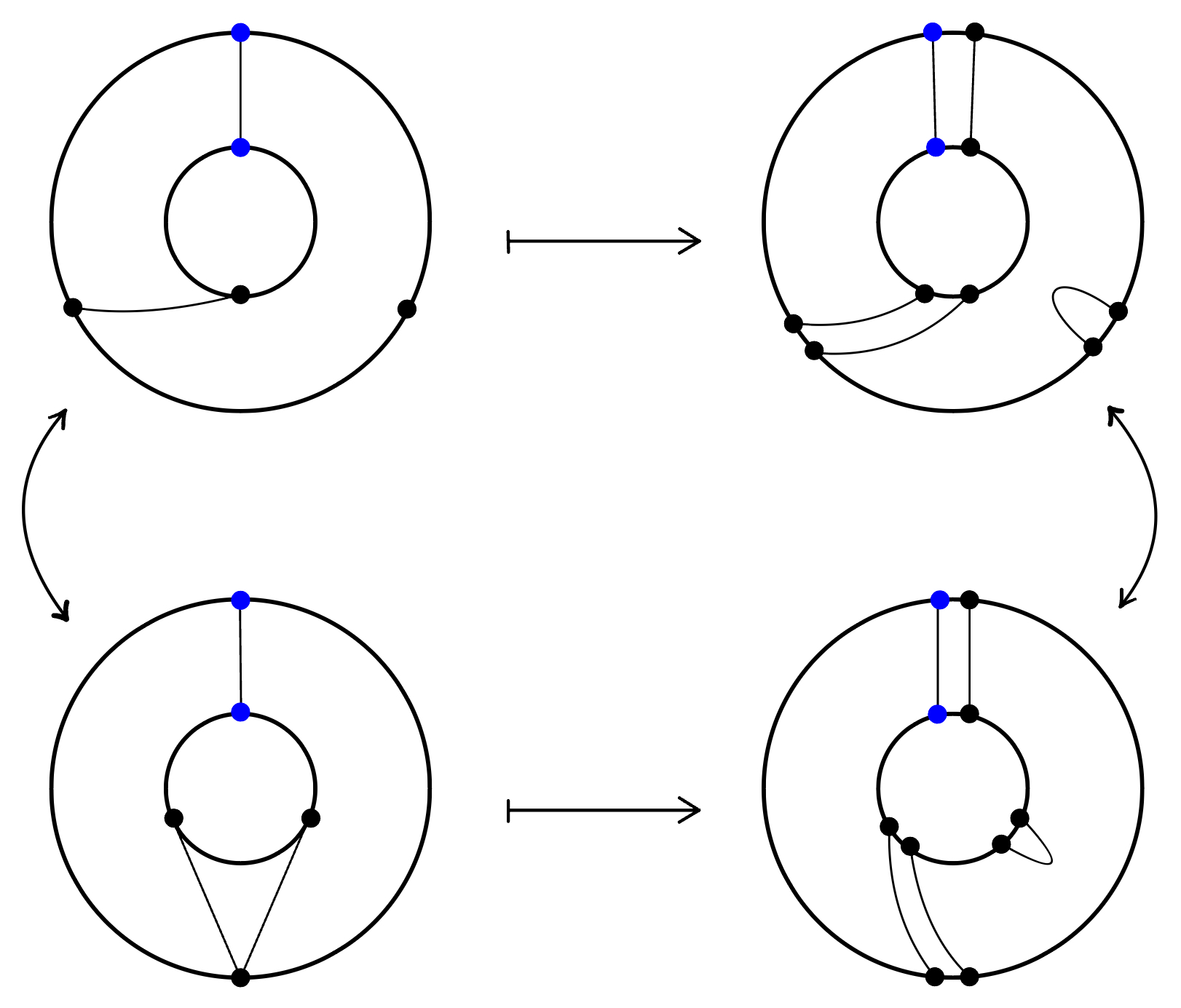}};

\draw (164,40.95) node  [font=\footnotesize]  {$0$};
\draw (164,114.95) node  [font=\footnotesize]  {$0$};
\draw (164,155.95) node  [font=\footnotesize]  {$1$};
\draw (251,179.95) node  [font=\footnotesize]  {$1$};
\draw (75,178.95) node  [font=\footnotesize]  {$2$};
\draw (578,181.95) node  [font=\footnotesize]  {$1$};
\draw (564,199.95) node  [font=\footnotesize]  {$1'$};
\draw (481,39.95) node  [font=\footnotesize]  {$0$};
\draw (501,39.95) node  [font=\footnotesize]  {$0'$};
\draw (482,115.95) node  [font=\footnotesize]  {$0$};
\draw (498,115.95) node  [font=\footnotesize]  {$0'$};
\draw (479,155.95) node  [font=\footnotesize]  {$1'$};
\draw (500,155.95) node  [font=\footnotesize]  {$1$};
\draw (422,203.95) node  [font=\footnotesize]  {$2$};
\draw (405,187.95) node  [font=\footnotesize]  {$2'$};
\draw (484,290.95) node  [font=\footnotesize]  {$0$};
\draw (499,290.95) node  [font=\footnotesize]  {$0'$};
\draw (484,365.95) node  [font=\footnotesize]  {$0$};
\draw (499,365.95) node  [font=\footnotesize]  {$0'$};
\draw (514,394.95) node  [font=\footnotesize]  {$1$};
\draw (504,405.95) node  [font=\footnotesize]  {$1'$};
\draw (482,405.95) node  [font=\footnotesize]  {$2$};
\draw (471,395.95) node  [font=\footnotesize]  {$2'$};
\draw (500,481.95) node  [font=\footnotesize]  {$1$};
\draw (482,481.95) node  [font=\footnotesize]  {$1'$};
\draw (165,290.95) node  [font=\footnotesize]  {$0$};
\draw (165,363.95) node  [font=\footnotesize]  {$0$};
\draw (189,395.95) node  [font=\footnotesize]  {$1$};
\draw (144,395.95) node  [font=\footnotesize]  {$2$};
\draw (165,481.95) node  [font=\footnotesize]  {$1$};
\draw (45,265.95) node  [font=\footnotesize]  {$dual$};
\draw (603,264.95) node  [font=\footnotesize]  {$dual$};
\end{tikzpicture}
}
\caption{Example of the map $\Lambda \hookrightarrow\cyl$.}
\label{fig:duals}
\end{figure}

\begin{remark}
    Since both $\cyl_0$ and $\Lambda$ are self-dual, the op in the theorem above is superfluous and is merely present to make the comparison more direct.
\end{remark}

Recall that a cyclic object in a category $\cat C$ is defined as a functor $X_\bullet\colon \Lambda^{\rm op}\to \cat C$. In light of the inclusion $\Delta\to \Lambda$, a cyclic object can be viewed as a simplicial object with extra structure, namely that the $n$-simplices have an automorphism $X_n\to X_n$ satisfying the cyclic relations. This automorphism specifies an action of the cyclic group $C_{n}$ on $X_{n-1}$. To extend a cyclic object to a functor out of $\cyl_0$, we need $X_{n-1}$ to actually have a $C_{2n}$-action. This ``extra structure'' we are looking for is described by the following category.

\begin{definition}\label{defn:sqrt Lambda}
    Let $\sqrt{\Lambda}^{\rm op}$ be the category with the same generators and relations as $\Lambda^{\rm op}$ (\cref{defn:cyclic cat}) except that that $t_n$ is replaced with a generator called $\sqrt{t_n}\colon [n]\to [n]$, and relations (iv)--(vi) are replaced by the following:\begin{itemize}
        \item[(iv)] $\sqrt{t_n}^{2(n+1)}=\id$,
       \item[(v)] $\sqrt{t_{n+1}}^2\circ s_n^j = s_n^{j+1}\circ \sqrt{t_n}^2$,
        \item[(vi)] $\sqrt{t_n}^2\circ d_{n+1}^i = d_{n+1}^{i+1}\circ \sqrt{t_{n+1}}^2$.
    \end{itemize} 
\end{definition}

The inclusion $\Lambda^{\rm op}\to \sqrt{\Lambda}^{\rm op}$ is the identity on the face and degeneracy maps, and sends $t_n$ to $\sqrt{t_n}^2$. A functor $X\colon \sqrt{\Lambda}^{\rm op}\to \cat C$ is not quite a cyclic object, but can be viewed as a cyclic object whose $C_n$-action ``has a square root;'' there is an inclusion $\Lambda^{\rm op}\to \sqrt{\Lambda}^{\rm op}$ which is the identity almost everywhere except $t_n\in \Lambda^{\rm op}$ is mapped to $\sqrt{t_n}^2\in \sqrt{\Lambda}^{\rm op}$.  

\begin{remark}
    The category $\sqrt{\Lambda}^{\rm op}$ is similar to, but notably different from, the $C_2$-twisted category $\Lambda_2^{\rm op}$ in \cite{BHM:93}. In particular, relations (v) and (vi) in \cref{defn:sqrt Lambda} are different than those in $\Lambda_2^{\rm op}$.
\end{remark}

\begin{theorem}\label{thm:inclusion of c2 cyc into cyl}
{The category $\sqrt{\Lambda}^{\rm op}$ is isomorphic to the subcategory of $\cyl_0$ on objects $S^1_k$ for $k>0$, generated by morphisms in \cref{defn: defn of gen} that do not have $S^1_0$ as source or target.}
\end{theorem}
\begin{proof}
Just as in \cref{thm:cyl and cyclic cat}, the assignment $[n]\mapsto S^1_{2n+1}$ for $n\geq 1$ on objects extends to an inclusion which is an isomorphism onto its image. The image is generated by the generators of $\cyl_0$ (\cref{cor:cyl rep data}) except for the ones mentioned above, namely $\de^i_2$, and $\mathbf{b}^i_0$ for $i=0,1$.
\end{proof}

This identification gives us an explicit way to build $\cyl_0$-objects from $\sqrt{\Lambda}^{\rm op}$-objects. Indeed, suppose $X$ is a $\sqrt{\Lambda}^{\rm op}$-object. We can extend $X$ to a $\cyl_0$-object $Y$ by setting $Y(S^1_{2n}) = X_{n-1}$ for $n\geq 1$. The only data missing is a choice of $X_{-1}:=Y(S^1_0)$, along with the structure maps $Y(\de^i_2)\colon X_0\to Y(S^1_0)$ for $i=0,1$ and $Y(\bi_0^i)\colon Y(S^1_0)\to X_0$ which satisfy:\begin{itemize}
    \item $Y(\de^i_2)\circ Y(\bi^i_0) = \id$ for $i=0,1$,
    \item $Y(\de^0_2) \circ \sqrt{t_1} =Y(\de^1_2)$ and $\sqrt{t_1} \circ Y(\bi^0_0)  = Y(\bi^1_0)$. 
\end{itemize}
In particular, it is enough to specify a map $Y(\de^0_2)\colon X_0\to Y(S^1_0)$ with a section $Y(\bi_0^0)$. Note that there is no condition on the composition $Y(\bi^i_0)\circ Y(\de^i_2)$.

\begin{corollary}\label{cor:cyl0 obj gives c2 cyclic}
    A $\cyl_0$-object is specified by the data of a functor $X\colon \sqrt{\Lambda}^{\rm op}\to \cat C$ along with an object $X_{-1}\in \cat C$ and a choice of augmentation $X_0\to X_{-1}$ with a section $X_{-1}\to X_0$. 
\end{corollary}

One example of a $\sqrt{\Lambda}^{\rm op}$-object is a variation on the edgewise subdivision $sd(X)$ of a cyclic object $X$. Edgewise subdivision is a general construction on simplicial objects, and its restriction to cyclic objects defines a functor into $\Lambda_2^{\rm op}$-objects (see \cite[Section 1]{BHM:93}). We will describe a similar construction on simplicial objects, called \textit{doubling}, whose restriction to cyclic objects defines a functor into $\sqrt{\Lambda}^{\rm op}$-objects. 

The double of a simplicial object has the same $n$-simplices as the edgewise subdivision, but different maps. Define $\delta\colon \Delta^{\rm op} \to \Delta^{\rm op}$ by $[n]\mapsto [2n+1]$, $d^i_n\mapsto d^{2i}_{2n}\circ d^{2i}_{2n+1}$, and $s^j_n\mapsto s^{2j}_{2n+2}\circ s^{2j}_{2n+1}$.

\begin{lemma}
    The assignment $\delta$ is functorial.
\end{lemma}
\begin{proof}
    We need to check that $\delta$ preserves the simplicial relations, i.e. is well-defined. For $i<j$, we check that relation (i) of \cref{defn:cyclic cat} is preserved: \begin{align*}
        \delta(d^i_n)\circ \delta(d^j_{n+1}) &= (d^{2i}_{2n}d^{2i}_{2n+1})\circ(d^{2j}_{2n+2}d^{2j}_{2n+3}) & \\
        &= d^{2j-2}_{2n} \circ (d^{2i}_{2n+1}d^{2i}_{2n+2}) \circ d^{2j}_{2n+3}\\
        &= (d^{2j-2}_{2n}d^{2j-2}_{2n+1})\circ(d^{2i}_{2n+2}d^{2i}_{2n+3})\\
        &= \delta(d_{n}^{j-1})\circ \delta(d^i_{n+1}),
    \end{align*} using the fact that $2i<2j, 2j-1$. Relation (ii) is similar. For relation (iii), we check the cases where $i=j,j+1$, as the other two cases are similar to the argument above. We have\begin{align*}
        \delta(d^i_{n+1})\circ \delta(s^j_{n}) &= (d^{2i}_{2n+2}d^{2i}_{2n+3}) \circ (s^{2j}_{2n+2}s^{2j}_{2n+1})
    \end{align*}
    which is clearly the identity for $i=j$. When $i=j+1$, we have\begin{align*}
        \delta(d^{j+1}_{n+1})\circ \delta(s^j_{n}) &= (d^{2j+2}_{2n+2}d^{2j+2}_{2n+3}) \circ (s^{2j}_{2n+2}s^{2j}_{2n+1})\\
        &= (d_{2n+2}^{2j+1} d_{2n+3}^{2j+2}) \circ (s_{2n+2}^{2j+1}s_{2n+1}^{2j}) & \text{by relations (i) and (ii)},\\
        &= d_{2n+2}^{2j+1} \circ \id_{2n+2}\circ s^{2j}_{2n+1} \\
        &= \id_{2n+1} = \delta(\id_n).
    \end{align*}
\end{proof}

\begin{definition}
    Let $X$ be a simplicial object and define the \textit{double of $X$} to be $db(X):= X\circ \delta$. 
\end{definition}

\begin{proposition}
    If $X$ is a cyclic object, then $db(X)$ is a $\sqrt{\Lambda}^{\rm op}$-object.
\end{proposition}\begin{proof}
   We can extend $\delta$ to a functor $\Lambda^{\rm op} \to \Lambda^{\rm op}$ by sending $t_n$ to $t^2_{2n+1}$, as this preserves the relation $t_n^{n+1}=\id_n$; the relations (v) and (vi) from \cref{defn:cyclic cat} are also preserved. It suffices to show that this extension of $\delta$ factors as \[
   \Lambda^{\rm op}\hookrightarrow \sqrt{\Lambda}^{\rm op}\to \Lambda^{\rm op}.
   \] We may define the map $\sqrt{\Lambda}^{\rm op}\to \Lambda^{\rm op}$ by $\delta$ on the face and degeneracy maps, and sends the generator $\sqrt{t_n}\colon [n]\to [n]$ to $t_{2n+1}\colon [2n+1]\to [2n+1]$. The relation $\sqrt{t_n}^{2(n+1)}=\id_n$ is preserved, as $\delta(\sqrt{t_{2n+1}}^{2(n+1)}) = t_{2n+1}^{2n+2} = \id_{2n+1}$. This is the claimed factorization of $\delta$.
\end{proof}

\begin{example}
    For instance, if $X=N^{cyc}(R)$ is the cyclic bar construction, then $db(X)_n = N^{cyc}_{2n+1}(R) = R^{\wedge 2n+2}$ has a natural $C_{2n+2}$-action by permuting the factors. The face maps $d^i$ multiply the $2i-1, 2i$, and $2i+1$ factors of $R^{\wedge 2n+2}$ (except for $d^n$ which incorporates the $C_{2n+2}$-action); the degeneracy maps $s^j$ insert the unit into the $2j-1$ and $2j$ factors of $R^{\wedge 2n+2}$ (except for $s^0$ which also incorporates the $C_{2n+2}$-action).
\end{example}

To extend $db(X)$ to a $\cyl_0$-object, there are a few options. Using the notation of \cref{cor:cyl0 obj gives c2 cyclic}, one option is to take $db(X)_{-1}= X_0$ and the structure maps to be the face and degeneracies between $X_0$ and $X_1$; another option is to take $db(X)_{-1}=X_1$ and the structure maps to be identities; a third option (if $\cat C$ has a zero object $*$) is to take $db(X)_{-1}=*$ and the structure maps to be the unique morphisms between $*$ and $db(X)_0 = X_1$.

\subsection{The $\cyl$-bar construction}\label{sec:bar construction}

In this subsection, we introduce an example of a $\cyl^a$-object called the \textit{$\cyl$-bar complex}, which we plan to study this bar construction further in future work. Our construction is inspired by the \textit{$C_2$-twisted cyclic bar complex} \cite[Definition 8.1]{ABGHLM:18}, which is used to construct $C_2$-twisted topological Hochschild homology of a ring spectrum with involution. Rather than taking in involutive ring objects as input, our bar construction is built for dualizable objects.

Suppose that $(\cat C, \otimes, I)$ is a \emph{strict} monoidal category and $X$ is a self-dual object in $\cat C$. This means that there exists an evaluation morphism $\varepsilon: X \otimes X \mapsto I$ and a coevaluation morphism $\eta: I \mapsto X \otimes X$ and these adhere to coherence diagrams (the ``snake relations''). We use $\eta$ and $\varepsilon$ to construct a $\cyl^a$-object in $\calc$ called $B_\bullet^\cyl(X)$.

\begin{definition}\label{defn:bar construction}
Define $B_n^\cyl(X) = X^{\otimes n}$, with $B_0^\cyl(X) = I$, together with maps 
\begin{itemize}
    \item[] $d^i_n:B_n^\cyl(X) \to B_{n-2}^\cyl(X)$ for $i = 0, \dots, {n-1}$;
    \item [] $s^i_n:B_n^\cyl(X) \to B_{n+2}^\cyl(X)$ for $i = 0, \dots, n+1$;
    \item[] $t_n: B_n^\cyl(X) \to B_n^\cyl(X)$
\end{itemize}
defined as follows.

The map $t_n$ is the $C_n$-action that cyclically permutes the factors $X^{\otimes n}$ to the right. The maps $d^i_n$ and $s^i_n$ are defined by means of the evaluation, coevaluation, and $t_n$, as follows:
\[
d^i_n =  \begin{cases}  
\varepsilon \otimes {\rm id} \otimes \dots \otimes {\rm id} & i=0\\
{\rm id} \otimes \dots \otimes \varepsilon \otimes \dots  \otimes {\rm id} & 0<i<{n-2}\\
{\rm id} \otimes \dots \otimes {\rm id} \otimes \varepsilon  & i=n-2\\
 ({\rm id} \otimes \varepsilon \otimes {\rm id} \otimes \dots \otimes {\rm id}) \circ t_n^2 
   & i=n-1\\
\end{cases}
\]

\[
s^i_n  =  \begin{cases}  
\eta \otimes {\rm id} \otimes \dots  \otimes {\rm id} & i=0\\
{\rm id} \otimes \dots \otimes \eta \otimes \dots  \otimes {\rm id} & 0<i\leq{n-1}\\
{\rm id} \otimes \dots   \otimes {\rm id}  \otimes \eta & i=n\\
t_{n+2} \circ ({\rm id} \otimes  \dots   \otimes {\rm id} \otimes \eta)\phantom{xi}
  & i=n+1,\\
\end{cases}
\] where we freely make use of $I$ as a two-sided unit for $\otimes$.
\end{definition}

Note that in order for the relations in \cref{cor:cyl rep data} to be satisfied on the nose, we need $\cat C$ to be a strict monoidal category. The appearance of $t_n^2$ in the description of $d_n^{n-1}$ might be surprising, but it is an artifact of requiring $\de_n^{n-1} = \tw_{n-2}\circ \de_n^{n-2} \circ \tw_n^{-1}$. 
Since we can construct the generators $\de_k^i$ and $\bi_k^j$ from $\de_k^0$ and $\bi_k^0$ respectively by means of conjugating with twists, it would also suffice to only define $d^0_n$, $s^0_n$ and $t_n$.

\begin{theorem}
For any self-dual object $X$, the $\cyl$-bar complex $B_\bullet^{\cyl}(X)$ is a $\cyl^a$-object in $\cat C$.
\end{theorem}
\begin{proof}
We claim that the assignment $S^1_n \mapsto B_n^\cyl(X)$, $\de^i_n \mapsto d^i_n$, $\bi^i_n \mapsto s^i_n$, $\tw_n \mapsto t_n$ satisfies the relations described in \cref{rmk:gens and relns for cyla}.
Most of the relations rely on keeping track of the factors and are straightforward to check. The only non-trivial relation is (iii) in the case when $i=j-1,j,j+1$. The fact that $d_{n+2}^{j-1}\circ s_n^j = \id = d_{n+2}^{j+1}\circ s_n^j$ is precisely the snake relations on evaluation and coevaluation. The composition $d_{n+2}^j\circ s_n^j$ is $\varepsilon \circ \eta\colon I\to I$ (the Euler characteristic or categorical dimension of the object $X$), which commutes with all other maps, so we can send $(id_n, 1)$ to $\varepsilon \circ \eta \otimes id_{B_n^\cyl(X)}$.
\end{proof}

As a more concrete example, consider $\cat C = \Vect_k$, the category of vector spaces over a field $k$. If $V$ is a finite-dimensional vector space, then a choice of an inner product on $V$ defines an isomorphism $V\xrightarrow{\cong} V^*$, and there are evaluation $\varepsilon\colon V \otimes V \mapsto I$ and coevaluation $\eta\colon I \mapsto V \otimes V$ maps that adhere to coherence relations. We can choose a basis $\{e_i \}$ for $V$ to make the associator and unitors be the identity morphism. In this case the evaluation is given by $\varepsilon (x_1, x_2) = \langle x_1, x_2 \rangle$, and $\eta (1) = \sum_i e^*_i \otimes e_i$, where $e^*_i$ is defined by $e^*_i (x) = \langle e_i, x \rangle$.

In this case, we have $B_n^\cyl(V) = V^{\otimes n}$, with $B_0^\cyl(V) = k$, together with maps $t_n \colon B_n^\cyl(V) \to B_{n}^\cyl(V)$,
$d^i_n\colon B_n^\cyl(V) \to B_{n-2}^\cyl(V)$ for $i = 0, \dots, {n-1}$, and maps 
$s^j_n: B_n^\cyl(V) \to B_{n+2}^\cyl(V)$ for $j = 0, \dots, n+1$ given on simple tensors by
\begin{flalign*}
   \hspace{1.4cm} &t_n (x_0 \otimes \dots  \otimes x_{n-1}) =   x_{n-1} \otimes x_0 \otimes \dots \otimes x_{n-2},
&
\end{flalign*}

\[
d^i_n (x_0 \otimes \dots  \otimes x_{n-1}) =  \begin{cases}  
\varepsilon(x_0, x_1) \otimes x_2 \otimes \dots  \otimes x_{n-1} & i=0\\
x_0 \otimes \dots \otimes \varepsilon(x_i, x_{i+1}) \otimes \dots  \otimes x_{n-1} & 0<i<{n-2}\\
x_0 \otimes \dots \otimes x_{n-3} \otimes  \varepsilon(x_{n-2}, x_{n-1}) & i=n-2\\
 x_{n-2} \otimes \varepsilon(x_{n-1}, x_0) \otimes x_1 \otimes \dots \otimes x_{n-3}  & i=n-1\\
\end{cases}
\]
and
\[
\hspace{0.4cm}  s^j_n (x_0 \otimes \dots  \otimes x_{n-1}) =  \begin{cases}  
\sum_i  e^*_i \otimes e_i \otimes x_0 \otimes \dots  \otimes x_{n-1} & j=0\\
\sum_i x_0 \otimes \dots \otimes  e^*_i \otimes e_i \otimes x_j \otimes \dots  \otimes x_{n-1} & 0<j\leq{n-1}\\
\sum_i x_0 \otimes \dots   \otimes x_{n-1}  \otimes  e^*_i \otimes e_i & j=n\\
\sum_i  e_i \otimes x_0 \otimes \dots   \otimes x_{n-1}  \otimes e^*_i
  & j=n+1.\\
\end{cases}
\]
The map $\chi_0 \colon k \xrightarrow{\eta} V\otimes V\xrightarrow{\varepsilon} k$ is multiplication by $\dim(V)$ and $\chi_n=\chi_0\cdot \id_n$.

\appendix
\section{Stratified Morse theory background}\label{sec: SMT}
The following is a summary of stratified Morse theory definitions and results. This material is drawn from \cite{SMT88}. Let $\cals$ be a partially ordered set; it will index the strata of the space $Z$.

\begin{definition}[\cite{SMT88}, I.1.1]\label{def: S-decomp}
An \emph{$\cals$-decomposition} of a topological space $Z$ is a locally finite collection of disjoint locally closed subsets $S_i\subset Z$ for each $i\in \cals$, such that
\begin{enumerate}
    \item $Z=\bigcup_{i\in\cals} S_i$
    \item $S_i\cap \overline{S_j}\neq \emptyset \Leftrightarrow S_i\subset\overline{S_j} \Leftrightarrow i=j$ or $i<j$
\end{enumerate}
\end{definition}
Let $Z$ be a closed subset of a smooth manifold $M$, and suppose $Z$ has an $\cals$-decomposition.

\begin{definition}[\cite{SMT88}, I.1.2]\label{def: Whit strat}
    The $\cals$-decomposition of $Z$ is a \emph{Whitney stratification} of $Z$ provided:
    \begin{enumerate}
        \item Each piece $S_i$ is a locally closed smooth submanifold (may or may not be connected) of $M$.
        \item Whenever $S_\alpha<S_\beta$ then the pair satisfies Whitney's conditions (a) and (b): suppose $x_i\in S_\beta$ is a sequence of points converging to some $y\in S_\alpha$. Suppose $y_i\in S_\alpha$ also converges to $y$, and suppose that the secant lines $l_i=\overline{x_i y_i}$ converge to some limiting line $l$, and the tangent planes $T_{x_i}S_\beta$ converge to some limiting plane $\tau$. Then
            \begin{enumerate}
                \item $T_y S_\alpha\subset \tau$ and
                \item $l\subset \tau$
            \end{enumerate}
    \end{enumerate}
\end{definition}

\begin{remark}
    Note that (2b) implies (2a).
\end{remark}

Fix a Whitney stratification of a subset $Z$ of a smooth manifold $M$. 
Suppose $p\in Z$ and let $S$ be the stratum of $Z$ which contains $p$.
\begin{definition}[\cite{SMT88}, I.1.8]\label{def: gen tangent space}
    A \emph{generalized tangent space} $Q$ at the point $p$ is any plane of the form $$Q=\lim_{p_i\to p} T_{p_i}R$$
    where $R>S$ is a stratum of $Z$ and $p_i\in R$ is a sequence converging to $p$. 
\end{definition}

Goresky--MacPherson define analogs of smooth functions, critical points and Morse functions, for the stratified setting. Then analogs of the main theorems for Morse theory will apply in the stratified setting as well.

\begin{definition}[\cite{SMT88}, I.2.1]\label{def: crit pt}
Fix a Whitney stratification of $Z\subset M$. Consider a smooth function $\Tilde{f}: M\to \bbr$ and its restriction $f:=\Tilde{f}|_Z:Z\to \bbr$.
A \emph{critical point of $f$} is any point $p\in S$ such that $d\Tilde{f}(p)|_{T_pS}=0$, where $S$ is the stratum of $Z$ containing $p$.

The corresponding critical value $v=f(p)$ is \emph{isolated} if there exists an $\epsilon>0$ such that $f^{-1}[v-\epsilon, v+\epsilon]$ contains no critical points other than $p$.
\end{definition}

\begin{definition}[\cite{SMT88}, I.2.1]\label{def: strat Morse func}
    A \emph{ (stratified) Morse function} $f:Z\to \bbr$ is the restriction of a smooth function $\Tilde{f}: M\to \bbr$ such that
    \begin{enumerate}
        \item $f$ is proper and the critical values of $f$ are distinct.
        \item For each stratum $S$ of $Z$, the critical points of $f|_S$ are nondegenerate.
        \item For every such critical point $p\in S$ and for each generalized tangent space $Q$ at $p$, $d\Tilde{f}_p(Q)\neq 0$ except for the single case $Q=T_pS$
    \end{enumerate}
\end{definition}

\begin{remark}
    Note that the critical points of a Morse function are isolated. 
\end{remark}

\begin{remark}
    Some intuition behind the definition of stratified Morse function: conditions (1) and (2) mean that the restriction of $f$ to each stratum of $Z$ is Morse in the classical sense. Condition (2) is a nondegeneracy requirement in the tangential directions to $S$, while condition (3) ensures that a critical point of the stratum $S$ is not also a limiting critical point for a higher stratum.
\end{remark}

\begin{theorem}[\cite{SMT88}, Theorem 2.2.1]\label{thm: SMT dense in smooth maps}
    Let $Z$ be a closed Whitney stratified subanalytic subset of an analytic manifold $M$. Then the functions $\Tilde{f}:M\to \bbr$ whose restriction $f:=\Tilde{f}|_Z$ are Morse functions form an open and dense subset of the space $C^\infty_p(M, \bbr)$ of smooth proper maps on $M$.
\end{theorem}

\begin{definition}[\cite{SMT88}, I.2.3]
   Let $Z$ be a Whitney stratified subanalytic subset of an analytic manifold $M$. Let $f: Z\to \bbr$ be the restriction of a smooth function $\Tilde{f}: M\to \bbr$, and let $p\in Z$ be a critical point of $f$ contained in the stratum $S$ of $Z$.
   The critical point $p\in Z$ is \emph{nondepraved} if:
   \begin{enumerate}
       \item the critical point $p$ is isolated,
       \item the restriction $f|_S$ has a nondepraved critical point at $p$ (I.e. let $p_i$ be a sequence of points converging to $p$; suppose the vectors $v_i=\frac{(p_i-p)}{|p_i-p|}$ converge to some limiting vector $v$; suppose the subspaces $\ker df(p_i)$ converge to some limiting subspace $\tau$; suppose that $v\notin \tau$. Then for all $i$ sufficiently large, $df(p_i)(v_i)\cdot (f(p_i)-f(p))>0$.), and
       \item for each generalized tangent space $Q$ at $p$, $d\Tilde{f}(p)(Q)\neq 0$ except for the single case $Q=T_pS$.
   \end{enumerate}
\end{definition}

Just as in classical Morse theory, one of the main theorems of stratified Morse theory describes how the topology of the stratified space $Z$ changes as one moves past critical points of $Z$.

Fix $\epsilon>0$ so that the interval $[v-\epsilon, v+\epsilon]$ contains no critical values of $f$ other than $v=f(p)$.
\begin{definition}[\cite{SMT88}, I.3.3]
    A pair $(A,B)$ of $\cals$-decomposed spaces is \emph{Morse data} for $f$ at $p$ if these is an embedding $h:B\to Z_{\leq v-\epsilon}$ such that $Z_{\leq v+\epsilon}$ is homeomorphic to the space $Z_{\leq v-\epsilon}\cup_B A$, where the homeomorphism preserves the $\cals$-decompositions.
\end{definition}

Suppose $f:Z\to \bbr$ is proper and the critical value $v=f(p)$ is isolated.
\begin{definition}[\cite{SMT88}, I.3.4]
    The \emph{coarse Morse data} for $f$ at $p$ is the pair of $\cals$-decomposed spaces
    \[ (A,B):= (Z\cap f^{-1}[v-\epsilon, v+\epsilon], Z\cap f^{-1}(v-\epsilon)), \]
    where $\epsilon>0$ is any number such that the interval $[v-\epsilon, v+\epsilon]$ contains no critical values other than $v=f(p)$.
\end{definition}

\begin{definition}[\cite{SMT88}, Definition 3.5.2]\label{def: local Md}
    Choose a $\delta>0$ such that $\partial B_\delta^M(p)$ is transverse to all the strata in $Z$ and none of the critical points of $f|_{B_\delta}$ have critical value $v$, except for the critical point $p$ (i.e. for any stratum $S\subset B_\delta$ and for any critical point $q$ of $f|_S$, $f(q)\neq v$ unless $q=p$); note that such a $\delta$ exists by Lemma 3.5.1 of \cite{SMT88}.
    The \emph{local Morse data} for $f$ at $p$ is the coarse Morse data for $f|_{B_\delta}$ at $p$, i.e. the pair
    \[ (B_\delta\cap f^{-1}[v-\epsilon, v+\epsilon], B_\delta\cap f^{-1}(v-\epsilon)). \]
\end{definition}

\begin{theorem}[\cite{SMT88}, Theorem 3.5.4]
    If $v=f(p)$ is an isolated critical value, then the local Morse data for $f$ at $p$ is Morse data. In other words, choosing an $\epsilon$ where $v$ is the only critical value in the interval $[v-\epsilon, v+\epsilon]$, then $Z_{\leq v+\epsilon}$ is obtained as a topological space from $Z_{\leq v-\epsilon}$ by attaching the space $A$ along the space $B$ (where $A, B$ are as in Definition~\ref{def: local Md}).
\end{theorem}

One of the main theorems of \cite{SMT88} gives a description of local Morse data $(A,B)$ in terms of tangential and normal Morse data. The latter requires the notion of the normal slice.

\begin{definition}\label{def: normal slice}
    Let $N'$ be a smooth submanifold of $M$ which is transverse to each stratum of $Z$, intersects the stratum $S$ in the single point $p$, and satisfied $\dim(S)+\dim(N')=\dim(M)$. Choose a Riemannian metric on $M$ and let $r(z)= |z-p|$ for each $z\in M$. Let $B_\delta(p)$ denote the closed ball $B_\delta(p)=\lbrace z\in M| r(z)\leq \delta\rbrace$, where $\delta$ is sufficiently small such that $\partial B_\delta(p)$ is transverse to each stratum of $Z$ and each stratum in $Z\cap N'$.
    The \emph{normal slice} $N(p)$ through the stratum $S$ at the point $p$
is the set
\[ N(p)=N'\cap Z\cap B_\delta(p).\]
\end{definition}

\begin{definition}\label{def: tang and norm Md}
    The \emph{tangential Morse data} for $f$ at $p$ is the local Morse for $f|_X$ at $p$.
    The \emph{normal Morse data} for $f$ at $p$ is the local Morse data for $f|_N$ at $p$.
\end{definition}

\begin{theorem}[\cite{SMT88}, I.3.7]\label{thm: SMT main thm}
    For a fixed stratification of $Z$ and a fixed function $f$ with a nondepraved critical point $p\in Z$, there is a $\cals$-decomposition preserving homeomorphism of pairs:
    Local Morse data $\cong$ (Tangential Morse data) $\times$ (Normal Morse data);

    i.e. if $(P,Q)$ is the tangential Morse data and $(J,K)$ is the normal Morse data, then the local Morse data is given by 
    \[ (P\times J, P\times K\cup Q\times J). \]
\end{theorem}

\bibliographystyle{amsalpha}
{\footnotesize\bibliography{refs}}

\end{document}